\RequirePackage{fix-cm}

\documentclass[smallextended]{svjour3}
\smartqed
\usepackage{graphicx}
\usepackage{graphicx,epsfig}
\usepackage{enumerate}
\usepackage{amssymb}

\usepackage{xcolor}
\usepackage{microtype}

\usepackage{mathrsfs}
\usepackage[utf8]{inputenc}
\usepackage{bm}
\usepackage{mathtools}
\usepackage{amsmath}
\usepackage{amsbsy}
\usepackage{mathrsfs}
\usepackage{cite}
\usepackage{float}
\usepackage{xcolor}
\usepackage{tikz}
\usetikzlibrary{positioning}
\usepackage{pgfplots}
\usepackage{pst-plot}
\usepackage{amsmath}
\usepackage{amsfonts}
\usepackage{amssymb}
 \usepackage[top=3.1cm,bottom=3.2cm,left=3.5cm,right=3.5cm]{geometry}

\definecolor{mypink1}{rgb}{0.858, 0.188, 0.478}
\definecolor{mypink2}{RGB}{219, 48, 122}
\definecolor{mypink3}{cmyk}{0, 0.7808, 0.4429, 0.1412}
\definecolor{mygray}{gray}{0.6}
\definecolor{mylightred}{RGB}{255,200,200}
\definecolor{mylightblue}{RGB}{30,144,255}
\definecolor{mylightgreen}{RGB}{0,153,0}

\smartqed
\usepackage[linkcolor=blue, urlcolor=blue, citecolor=blue,colorlinks, bookmarks]{hyperref}
\journalname{}

\begin{document}

\title{Generalized Hukuhara Weak Subdifferential and its Application on Identifying Optimality Conditions for Nonsmooth Interval-valued Functions
}

\titlerunning{$gH$-Weak Subdifferential Calculus
and its Application on IOPs
}

\author{Suprova Ghosh \and Debdas Ghosh \thanks{Debdas Ghosh (Corresponding Author)}}

\authorrunning{S. Ghosh and D. Ghosh}

\institute{Suprova Ghosh (Email: suprovaghosh.rs.mat19@itbhu.ac.in) \at Department of Mathematical Sciences, Indian Institute of Technology (Banaras Hindu University) Uttar Pradesh  221005, India
\and
Debdas Ghosh (Email: debdas.mat@iitbhu.ac.in) \at Department of Mathematical Sciences, Indian Institute of Technology (Banaras Hindu University) Varanasi, Uttar Pradesh  221005, India
\and
}

\date{Received: date / Accepted: date}

\maketitle

\begin{abstract}
In this article, we introduce the idea of $gH$-weak subdifferential for interval-valued functions (IVFs) and show how to calculate $gH$-weak subgradients. It is observed that a nonempty $gH$-weak subdifferential set is closed and convex. In characterizing the class of functions for which the $gH$-weak subdifferential set is nonempty, it is identified that this class is the collection of $gH$-lower Lipschitz IVFs. In checking the validity of sum rule of $gH$-weak subdifferential for a pair of IVFs, a counterexample is obtained, which reflects that the sum rule does not hold. However, under a mild restriction on one of the IVFs, one-sided inclusion for the sum rule holds. Next, as applications, we employ $gH$-weak subdifferential to provide a few optimality conditions for nonsmooth IVFs. Further, a necessary optimality condition for interval optimization problems with difference of two nonsmooth IVFs as objective is established. Lastly, a necessary and sufficient condition via augmented normal cone and $gH$-weak subdifferential of IVFs for finding weak efficient point is presented.

\end{abstract}

\section{Introduction}

Interval arithmetic of  Moore \cite{Moore1966} is the foundation stone in interval analysis. Realistic applicability of Moore's method is relevant till today. We can currently discover several papers in the community of interval-valued optimization problems (IOPs) where Moore's interval analysis is applied extensively. To find optimality conditions for IOPs, ideas of derivatives for interval-valued function (IVF)  have been proposed \cite{chalco2013, markov, Osuna,  stefanini2009,Ghosh2019derivative}. In \cite{markov}, the concept of $gH$-differentiability for IVFs was introduced.
%Stefanini \cite{stefanini2009} used the idea of $gH$-derivative on interval differential equations.
Chalco-Cano \cite{chalco2019} addressed the algebraic property of $gH$-differentiable interval-valued functions. Ghosh et al. \cite{Ghosh2019derivative} proved the existence of $gH$-directional derivative for convex IVFs and presented optimality conditions for IOPs. \\

It is a familiar fact that in nonsmooth optimization, the classical gradient algorithm fails:  even in finding the optimum point, as there is no derivative, the conventional optimality condition $\nabla f(x)=0$ does not hold. More crucially, it is observed that optima of an almost everywhere differentiable function categorically arise at nondifferentiable points---for instance, take the minimization of $f(x) = |x|$. The notion of subdifferential, defined by Rockafellar \cite{rockafellar}, is a crucial factor in the body of optimization theory that perfectly replaces the role of the gradient to identify optima for convex functions. However, subdifferential is inadequate in developing optimality conditions for nonconvex optimization problems. Due to this insufficiency, the idea of subdifferential has been generalized. Most common of such generalizations is weak subdifferential \cite{azimov}. Based on this notion, a strong duality theorem for nonconvex inequality constrained problem has been found by defining a weak conjugate function \cite{Yalcin}. A substantial application of this notion in duality theory with the help of a weak subdifferentiable perturbation function is given in \cite{gasimov}. \\

In the context of nonsmooth calculus for nondifferentiable convex IVFs, Ghosh et al.  \cite{Ghosh2020lasso} has recently proposed the idea of $gH$-subgradient and $gH$-subdifferential. The same article \cite{Ghosh2020lasso} found that $gH$-directional derivative is the maximum of all the products of the direction and $gH$-subgradients. Afterward, Anshika et al. \cite{anshika} characterized weak efficiency for nonconvex composite optimization problems with the subdifferential sets of convex interval-valued functions. In \cite{anshika}, by formulating supremum and infimum of an IVF, a Fermat-type, a Fritz-John-type, and a KKT-type weak efficiency condition for nonsmooth IOPs have been derived. Anshika et al.  \cite{anshika2} introduced $gH$-subdifferential of interval-valued value function.  Furthermore, Chauhan et al.  \cite{chauhan2021generalized} derived the notion of $gH$-Clarke derivative for IVFs and IOPs. Under the Clarke subdifferentiablility assumption, Chen and Li \cite{chen}  provided KKT conditions for efficient solutions. Additionally, Karaman et al. \cite{karaman} presented two subdifferentials for interval-valued functions and some optimality criteria, which were obtained by using subdifferentials. \\

From the available literature on nonsmooth IOPs, it is found that the study of $gH$-weak subdifferential notion has not yet been addressed. However, the notion of $gH$-weak subdifferential might be effective to characterize and capture the efficient solutions of IOPs with nonconvex and nonsmooth IVFs. By using subgradient, one may face difficulties to solve problem which does not satisfy convexity assumption because  subgradient refers to the slope of a supporting hyperplane to the graph of convex functions in convex analysis. Thus, in this study, we introduce the notion of weak subgradient, which does not need any kind  of convexity.  \\

In this article, we attempt to show various properties of weak-subdifferential and its use in nonsmooth nonconvex IOPs. As an application of the proposed $gH$-weak subdifferential, we propose  a necessary  and sufficient optimality condition for finding weak efficient points of difference of two IVFs.  \\

The rest of article is presented as follows. Section \ref{section2} is devoted to the conventional properties of intervals, followed by calculus of IVFs.  Section \ref{section3}  introduces the notion of $gH$-weak subdifferential for IVFs and discusses their properties such as convexity, closedness and nonemptiness. Additionally, the role of $gH$-weak subdifferential to derive the necessary condition for weak efficiency for $gH$-weak subdifferentiable IVFs is shown in Section \ref{section3}. In Section \ref{section5}, we analyze the necessary condition for obtaining efficient solution of difference of two IVFs. Finally, we draw conclusion with future directions to extend the present study.

\section{\textbf{Preliminaries}} \label{section2}
\noindent In this section, required terminologies and notions on intervals including calculus of IVFs are given. Throughout the paper, we extensively use the following notations.
	
	\begin{itemize}
	\item $\mathbb{R}$ is the set of real numbers
    \item $\mathbb{R}_+$ represents the set of nonnegative real numbers
    \item $I(\mathbb{R})$ is the collection of all compact intervals
    \item $\overline{I(\mathbb{R})}=I(\mathbb{R})\cup \{{-\infty, +\infty}\}$
     \item $\textbf{0}=[0,0]$
    \item Elements of $I(\mathbb R)$ are presented by bold capital letters:   $\textbf X, \textbf Y, \textbf Z,  \ldots$
    \item $I(\mathbb{R})^n=I(\mathbb{R}) \times I(\mathbb{R})\times I(\mathbb{R})\times \cdots \times I(\mathbb{R})$ ($n$ times)
    \item  Interval vectors in $I(\mathbb R)^n$ are denoted by $\widehat{\textbf X}, \widehat{\textbf Y}, \widehat{\textbf Z}, \ldots$
    \item $B_{\alpha}(\bar u)$ is the open ball with center at $\bar u \in \mathbb R^n$ and radius $\alpha \geq 0$
    \item $\mathcal{N}(\bar x)$ is a neighbourhood of $\bar x \in \mathbb R^n$
    \item $\lVert \cdot \rVert_{I(\mathbb R)}$ denotes the norm on $I(\mathbb R)$
    \end{itemize}

\subsection{Arithmetic and  dominance of intervals}
%%%%%%%%%%%%%%%%%%%%%%%%%%%%%%%%%%%%%%%%

Throughout the text, we represent an element $\textbf{X}$ of $I(\mathbb{R})$ by the corresponding small letter:
	\[\textbf{X} = [\underline{x}, \overline{x}],~\text{where}~\underline{x}~\text{and}~\overline{x}~ \text{are in} ~\mathbb{R}~\text{with} ~\underline{x} \leq \overline{x}.\]

\noindent
%It is to mention that any singleton element $\{p\}$ of $\mathbb{R}$  can be expressed by an interval $\textbf{P}=[\underline{p},\overline{p}]$, where $\underline{p}=p=\overline{p}$.\\ \\
%%%%%%%%%%%%%%%%%%%%%%%%%%%%%%%%%%%%%%%%%%%%%%%%%%%%%%%%%%%%%%%%%%%%%%%
Recall that Moore's interval addition $(\oplus)$, subtraction $(\ominus)$, and multiplication $(\odot)$  \cite{Moore1966,Moore1979} are given by
	  %followed by the concept of $gH$-difference of two intervals and also discuss the ordering of intervals \cite{Ishibuchi1990}. \\
	\begin{align*}
	  &  \textbf{X} \oplus \textbf{Y} = \left[~\underline{x} + \underline{y},~ \overline{x} +
	\overline{y}~\right], ~\textbf{X} \ominus \textbf{Y} = \left[~\underline{x} - \overline{y},~ \overline{x} -
	\underline{y}~\right], ~\text{and}~\\
	  &  \textbf{X} \odot \textbf{Y}  = \left[\min\left\{\underline{x}\; \underline{y},\; \underline{x}\overline{y},\; \overline{x}\underline{y},\; \overline{x}\overline{y}\right\},\; \max\left\{\underline{x}\; \underline{y},\; \underline{x}\overline{y},\; \overline{x}\underline{y},\; \overline{x}\overline{y}\right\} \right].
	\end{align*}
	\begin{definition}(\emph{$gH$-difference of intervals}  \cite{stefanini2008})\label{gh_difference}.  The $gH$-difference
		for a pair of intervals $\textbf{P}$ and $\textbf{Q}$, denoted by  $\textbf{P} \ominus_{gH} \textbf{Q}$, is the interval $\textbf{Y}$ such that
		\[\textbf{P} =  \textbf{Q} \oplus  \textbf{Y}~\text{or}~\textbf{Q} = \textbf{P} \ominus \textbf{Y}.\]
		It is well-known that for $\textbf{P} = \left[\underline{p},\overline{p}\right]$ and $\textbf{Q} = \left[\underline{q},\overline{q}\right]$,
		\[
		\textbf{P} \ominus_{gH} \textbf{Q} = \left[\min\{\underline{p}-\underline{q},
		\overline{p} - \overline{q}\},  \max\{\underline{p}-\underline{q}, \overline{p} -
		\overline{q}\}\right] \text{ and } \textbf{P} \ominus_{gH} \textbf{P} = \textbf{0}.
		\]
	\end{definition}
For two elements $\widehat{\textbf{I}}=(\textbf{I}_1,\textbf{I}_2,\ldots,\textbf{I}_n)$ and $\widehat{\textbf{J}}=(\textbf{J}_1,\textbf{J}_2,\ldots,\textbf{J}_n)$ of $I(\mathbb R)^n$, the algebraic operation $\widehat{\textbf{I}} \boldsymbol{\star}\widehat{\textbf{J}}$ is defined by
	\[
	 \widehat{\textbf{I}}\boldsymbol{\star}\widehat{\textbf{J}}=(\textbf{I}_1\boldsymbol{\star}\textbf{J}_1,\textbf{I}_2\boldsymbol{\star}\textbf{J}_2,\ldots,\textbf{I}_n\boldsymbol{\star}\textbf{J}_n),
	\]
	where $\boldsymbol{\star} \in \{\oplus,\ominus,\ominus_{gH}\}$.
\begin{definition}\label{interval_dominance}
	 	(\emph{Dominance of intervals} \cite{wu2007karush}). Let $\textbf{Z}$ and $\textbf{W}$ be in $I(\mathbb{R})$.
	 	\begin{enumerate}[(i)]
	 		\item $\textbf{W}$ is called \emph{dominated} by $\textbf{Z}$ if $\underline{z}~\leq~ \underline{w}$ and $\overline{z}~\leq~\overline{w}$, and then we express it by  $\textbf{Z}~\preceq~ \textbf{W}$.
	 		\item $\textbf{W}$ is said to be \emph{strictly dominated} by $\textbf{Z}$ if either `$\underline{z} ~\leq~ \underline{w}$  and $\overline{z} ~<~ \overline{w}$' or `$\underline{z} ~<~ \underline{w}$  and $\overline{z} ~\leq~ \overline{w}$', and then we express it by $\textbf{Z}~\prec~ \textbf{W}$.
	 		\item If $\textbf{W}$ is not dominated by $\textbf{Z}$, then we write $\textbf{Z}~\npreceq~ \textbf{W}$. If $\textbf{W}$ is not strictly dominated by $\textbf{Z}$, then we write $\textbf{Z}~\nprec~ \textbf{W}$.
	 		\item If $\textbf{W}~\npreceq~ \textbf{Z}$ and $\textbf{Z}~\npreceq~ \textbf{W}$, then it is called that \emph{none of $\textbf{W}$ and $\textbf{Z}$ dominates the other}, or $\textbf{W}$ and $\textbf{Z}$ \emph{are not comparable}.
	 	\end{enumerate}
	 \end{definition}
For any two elements $\widehat{\textbf{I}}=(\textbf{I}_1,\textbf{I}_2,\ldots,\textbf{I}_n)^\top$ and $\widehat{\textbf{J}}=(\textbf{J}_1,\textbf{J}_2,\ldots,\textbf{J}_n)^\top$ in $I(\mathbb{R})^n$,
\[\widehat{\textbf{I}} \preceq \widehat{\textbf{J}}\iff \textbf{I}_j \preceq \textbf{J}_j~\text{for all}~j=1, 2, \ldots, n.\]

\subsection{Concavity and differential calculus of IVFs}
 Let $\emptyset \neq \mathcal{Y} \subseteq \mathbb{R}^n$. Let an IVF ${\bf{\Phi}}:\mathcal{Y} \to I(\mathbb{R})$ be presented by
	\[{{\bf{\Phi}}}(y)=\left[\underline{\phi}(y),\overline{\phi}(y)\right] ~\forall~ y \in \mathcal{Y}, \]
    where $\underline{\phi}(y)\leq \overline{\phi}(y)$ for all $y \in \mathcal{Y}$ and $\underline{\phi}$ and $\overline{\phi}$ are called lower and upper real-valued functions on $\mathcal{Y}$ .\\
%%%%%%%%%%%%%%%%%%%%%  	
\begin{definition} (\emph{Concave IVF}). If $\mathcal{Y}$ is convex, then an IVF
		${\bf{\Phi}}$ is said to be a \emph{concave} IVF on $\mathcal{Y}$ if {for any} $y_1, y_2 \in \mathcal{Y},\beta_1,~\beta_2\in[0,\ 1],~\text{and}~\beta_1+\beta_2=1$,
	\begin{align*}
	\beta_1\odot{\bf{\Phi}}(y_1)\oplus\beta_2\odot{\bf{\Phi}}(y_2)\preceq{\bf{\Phi}}(\beta_1 y_1+\beta_2 y_2). \end{align*}
		\end{definition}
	\begin{lemma}\label{cr1}
	If $\bf{\Phi}$ is a concave IVF on a convex set $\mathcal{Y} \subseteq \mathbb{R}^n$, then $\underline{\phi}$
		and $\overline{\phi}$ are concave on $\mathcal{Y}$ and vice-versa.
	\end{lemma}
\begin{proof}
The proof is similar with the proof of Proposition 6.1 in \cite{wu2007karush}.
\end{proof}	
\begin{example}
 Let $\mathcal Y$ be the Euclidean space $\mathbb R^n$. Then, the IVF $\bf{\Phi}: \mathcal Y \to I(\mathbb R)$ which is defined by
 \begin{align*}
     {{\bf{\Phi}}}(y)= \widehat{\textbf M}^\top \odot y \ominus_{gH} \lVert y \rVert, ~\text{where}~ \widehat{\textbf M}=(\textbf M_1, \textbf M_2,\ldots, \textbf M_{n})\in I(\mathbb R)^n,
     \end{align*}
     \\ and for all~$ y =(y_1, y_2,\ldots, y_n)\in \mathcal Y$ is a concave IVF on $\mathcal Y$. The reason is as follows.\\

     Without loss of generality, first $p$ components of $y$ are assumed to be non-negative and rest $n-p$ be negative. Then, letting $\textbf M_{i}= [\underline m_i, \overline m_i]$ for all $i= 1,2,\ldots,n$,
      \begin{align*}
     {{\bf{\Phi}}}(y)= \bigoplus_{i=1}^p ~[\underline m_i y_i,\overline m_iy_i ]\oplus \bigoplus_{j=p+1}^n[\overline m_j y_j,\underline m_jy_j ]\ominus_{gH} \lVert y \rVert.
     \end{align*}
It is evident that $\sum _{i=1}^p \underline m_{i} y_{i}+ \sum_{j=p+1}^n \overline m_{j} y_{j}$ and $ \sum _{i=1}^p \overline m_{i} y_{i}+ \sum_{j=p+1}^n \underline m_{j} y_{j}$, being linear, are concave functions. Also, $- \lVert y \rVert$ is a concave function. Therefore,
 $\sum _{i=1}^p \underline m_{i} y_{i}+ \sum_{j=p+1}^n \overline m_{j} y_{j}-\lVert y \rVert$ and $ \sum _{i=1}^p \overline m_{i} y_{i}+ \sum_{j=p+1}^n \underline m_{j} y_{j}-\lVert y\rVert$    are  concave functions.  Hence, by Lemma \ref{cr1}, ${\bf{\Phi}}$ is a concave IVF.
 \end{example}

\begin{definition}(\emph{$gH$-continuity} \cite{ Ghosh2016newton}).
An IVF ${\bf{\Phi}}$ is said to be $gH$-\emph{continuous}  at $u \in \mathcal{Y}$ if
		\[
		\lim_{\lVert d \rVert\rightarrow 0}\left({\bf{\Phi}}(u+d)\ominus_{gH}{\bf{\Phi}}(u)\right)=\textbf{0}.\]
	 If at every $u \in \mathcal{Y}$, ${\bf{\Phi}}$ is $gH$-continuous, then ${\bf{\Phi}}$ is called $gH$-\emph{continuous} on $\mathcal{Y}$.
\end{definition}
 %%%%%%%%%%%%

	%%%%%%%%%%%%%%%%%

	\begin{lemma}\emph{(See \label{lc2}\cite{Ghosh2019gradient}).}
	 For a $gH$-continuous IVF ${\bf{\Phi}}$ its
		$\underline{\phi}$ and $\overline{\phi}$ are continuous and vice-versa.
	\end{lemma}
	%%%%%%%%%%%%%%%%%%%%%%%%%%%%%%%%%%%%%%%%%%%%%%%%%%%%%%%%%%%%%%%%%%%%%%%%%%%%%

	\begin{definition}(\emph{$gH$-derivative} \cite{chalco2012}). \label{derivative}
	Let $\mathcal{Y}\subseteq \mathbb{R}^n$. The \emph{$gH$-derivative} of an IVF ${\bf{\Phi}}: \mathcal{Y} \to I(\mathbb{R})$ at $u \in \mathcal{Y}$ is the limit
	\[
	{\bf{\Phi}}'(u) : =\underset{d \rightarrow 0} \lim ~\tfrac{1}{d} \odot\left\{{\bf{\Phi}}(u+d)\ominus_{gH}{\bf{\Phi}}(u)\right\}.\]
	\end{definition}
	%%%%%%%%%%%%%%%%%%%%%%%%%%%%%%%%%%%%%%%%%%%%%%%%%%%%%%%%%%%%%%%%%%%%%%%%%%%%%
	\begin{definition}(\emph{$gH$-G\'ateaux derivative} \cite{Ghosh2019derivative}). Let an IVF ${\bf{\Phi}}$ be defined on
	  a nonempty open subset $\mathcal{Y}$ of $\mathbb{R}^n$. Then, ${\bf{\Phi}}$ is known to be \emph{$gH$-G\'ateaux differentiable} with  \emph{$gH$-G\'ateaux derivative} ${\bf{\Phi}}_{\mathscr G}(u)$ at $u \in \mathcal{Y}$ if the following limit
	\[{\bf{\Phi}}_{\mathscr G}(u)(h):= \lim\limits_{\beta \to 0+}\frac{1}{\beta} \odot({\bf{\Phi}}(u+\beta h)\ominus_{gH}{\bf{\Phi}}(u))\]
	  is finite for all $h \in \mathbb{R}^n$ and ${\bf{\Phi}}_{\mathscr G}(u)$ is a $gH$-continuous and linear IVF from $\mathbb{R}^n$ to $I(\mathbb{R})$.
	  %then $\textbf F_{\mathscr G}(\bar x)$ is said to be $gH$-G\'ateaux derivative of $\textbf F$ at $\bar x$. If $\textbf F$ has $gH$-G\'ateaux derivative  at $\bar x$, then $\textbf F_{\mathscr G}(\bar x)$ is said to be $gH$-G\'ateaux differentiable at $\bar x$.
 	\end{definition}
	\begin{definition}(\emph{gH-Fr\'echet derivative} \cite{Ghosh2019derivative}). Let an IVF ${\bf{\Phi}}$ be defined on
	  a nonempty open subset $\mathcal{Y}$ of $\mathbb{R}^n$. Then,  ${\bf{\Phi}}$ is said to be \emph{$gH$-Fr\'echet differentiable} %with  $gH$-G\'ateaux derivative $\textbf T_{\mathscr G}(u)$
	  at $u \in \mathcal{Y}$   if there exists a $gH$-continuous and linear mapping  $\textbf G: \mathcal Y\rightarrow I(\mathbb{R})$ such that
	\[ \lim_{\substack{%
	\lVert h\rVert \to 0}}\tfrac{1}{\lVert h \rVert} \odot {\left( \lVert{\bf{\Phi}} (u+h)\ominus_{gH} {\bf{\Phi}}(u) \ominus_{gH} \textbf G(h)\rVert_{I(\mathbb{R})}\right)}=0,\]
	 where $\textbf G$ will be referred to as ${\bf{\Phi}}_{\mathscr F}(u)$.
	\end{definition}	

\begin{definition}\label{efficient_point_def}(\emph{Efficient point} \cite{Ghosh2019derivative}).
		Let $\mathcal{Y} \subseteq \mathbb{R}^n$ and ${\bf{\Phi}}: \mathbb{R}^n \rightarrow I(\mathbb{R})$ be an IVF. A point $u\in \mathcal{Y}$ is said to be an \emph{efficient point} of the IVF ${\bf{\Phi}}: \mathcal Y \rightarrow I(\mathbb{R})$
		%\begin{equation}\label{IOP}
		%\displaystyle \min_{x \in \mathcal{X}} \textbf{F}(x)
		%\end{equation}
		if ${{\bf{\Phi}}}(y)\nprec{\bf{\Phi}}(u)$ for all $y\in \mathcal{Y}.$
	\end{definition}

\begin{definition}\label{Weak efficient_point_def}(\emph{Weak efficient point} \cite{anshika}).
Let $\mathcal{Y} \subseteq \mathbb{R}^n$ and ${\bf{\Phi}}: \mathbb{R}^n \rightarrow I(\mathbb{R})$ be an IVF. A point $u\in \mathcal{Y}$ is said to be a\emph{ weak efficient point} of the IVF ${\bf{\Phi}}: \mathcal Y \rightarrow I(\mathbb{R})$
		%\begin{equation}\label{IOP}
		%\displaystyle \min_{x \in \mathcal{X}} \textbf{F}(x)
		%\end{equation}
		if ${\bf{\Phi}}(u)\preceq{\bf{\Phi}}({y})$ for all $y\in \mathcal{Y}.$
	\end{definition}
\subsection{Few properties of the elements in $I(\mathbb{R})$}

Let $\textbf Y=[\underline y, \overline y]$ and $\widehat{\textbf Y}=(\textbf Y_{1}, \textbf Y_{2}, \ldots, \textbf Y_{n})$ be  elements in $I(\mathbb{R})$ and $I(\mathbb{R})^n$, respectively. The following two functions $\lVert \cdot\rVert_{I(\mathbb{R})}: I(\mathbb{R}) \rightarrow \mathbb{R}_{+}$ and $\lVert \cdot\rVert_{I(\mathbb{R})^n}: I(\mathbb{R})^n \rightarrow \mathbb{R}_{+}$ are referred to as norm \cite{Moore1966,Moore1979} on $I(\mathbb{R})$ and $I(\mathbb{R})^n$, respectively:
\[ \lVert \textbf Y\rVert_{I(\mathbb{R})}= \max\{\lvert \underline y\rvert, \lvert \overline y\rvert\},~\text{and}~\lVert \widehat{\textbf Y} \rVert_{I(\mathbb{R})^n} =\sum\limits_{j=1}^{n} \lVert \textbf Y_{j} \rVert_{I(\mathbb{R})}.
\]

%\begin{definition}
%(\emph{Maximum and minimum of intervals} \cite{anshika}). Let $\textbf{X}_1,\textbf{X}_2,\ldots,\textbf{X}_m$ be the elements of $I(\mathbb{R})$ with $\textbf{X}_1\preceq \textbf{X}_2 \preceq \cdots \preceq \textbf{X}_m$. Then,
%	\[
%	\max\{\textbf{X}_i: i=1, 2, \ldots, m\} = \textbf{X}_m~\text{and}~\min\{\textbf{X}_i: i=1, 2, \ldots, m\} = \textbf{X}_1.\]
	 %i.e., \\
%	For any subset \textbf{S}= $\bigg\{[\underline{a}_{\mu},\overline{a}_{\mu}] \in I(\mathbb{R}):\mu \in \Lambda~\text{and}~\Lambda~\text{ is an index set} \bigg\}$ whose interval elements are comparable. Then we have $\max \textbf{S} = \bigg[\underset{\Lambda}{\text{max}}\underline{a}_{\mu},\underset{\Lambda}{\text{max}}\overline{a}_{\mu}\bigg]$ and $\min \textbf{S} = \bigg[\underset{\Lambda}{\text{min}}\underline{a}_{\mu},\underset{\Lambda}{\text{min}}\overline{a}_{\mu}\bigg]$.

%\end{definition}

%\begin{rmrk}
%It can be easily noted that the maximum of any set $\textbf{S}$ of $I(\mathbb{R})$ always belong to the set $\textbf{S}$.
%\end{rmrk}

%%%%%%%%%%%%%%%%%%%%%%%%%%%%%%%%%%%%%%%%%%%%%%%%%%%%%%%%%%%%%%%%%%%%%%%%%%%%%%%%%%	

		%\begin{lemma} \cite{Ghosh2019derivative}\label{dr1} Let $\textbf{X}$ $\textbf{Y}$ and $\textbf{Z}$ be two elements of $I(\mathbb{R})$.
		%\begin{enumerate}[(i)]
			%\item \label{part1_4}
		%	If $\textbf{Z} \preceq \textbf{Y}$, then $\textbf{X}\ominus_{gH} \textbf{Y} \preceq \textbf{X}\ominus_{gH} \textbf{Z}$.
		%\end{enumerate}
%	\end{lemma}

%%%%%%%%%%%%%%%%%%%%%%%%%%%%%%%%%%%%%%%%%%%%%%%%%%%%%%%%%%%%%%%%%%%%%%%%%%%%%%%%%%
\begin{lemma} \label{dr2}
For any $\textbf{W}, \textbf{Y}, \textbf{Z} \in I(\mathbb{R})$ and $ \epsilon \geq 0$, we have
\[ \epsilon \preceq (\textbf{W} \ominus_{gH} \textbf{Y}) \ominus_{gH} \textbf{Z}  \implies \textbf Z \oplus \epsilon \preceq\textbf{W} \ominus_{gH} \textbf{Y}. \]
\end{lemma}
\begin{proof}
See Appendix \ref{appendix_ind}.
\end{proof}
%%%%%%%%%%%%%%%%%%%%%%%%%%%%%%%%%%%%%%%%%%%%%%%%%%%%%%%%%%%%%%%%%%%%%%%%%%%%%%%%%%%%%%%%%
\begin{lemma} \label{gh2}
For any $\textbf X, \textbf Y, \textbf Z, \textbf W \in I(\mathbb{R})$, we have
\begin{align*}
(\textbf{X} \oplus \textbf Y) \ominus_{gH} (\textbf Z\oplus \textbf W) \subseteq  (\textbf X\ominus_{gH} \textbf Z) \oplus (\textbf Y \ominus_{gH} \textbf {W}).
\end{align*}
%\[\textbf G \preceq (\textbf X\oplus \textbf Y)\ominus_{gH} (\textbf Z \oplus \textbf W) \implies \textbf G  \preceq (\textbf X \ominus_{gH} \textbf Z) \oplus (\textbf Y \ominus_{gH} \textbf W)\]
\end{lemma}
\begin{proof}
See Appendix \ref{appendix_inf}.
\end{proof}
%%%%%%%%%%%%%%%%%%%%%%%%%%%%%%%%%%%%%%%%%%%%%%%%%%%%%%%%%%%%%%%%%%%%%%%%%%%%%%%%%%%%%%%%%%%%%
 \begin{lemma} \label{ldr1}
 %Let
%For all $\textbf{X}, \textbf{Y}, \textbf{Z} \in I(\mathbb{R})$,
%\[ (-1 \odot \textbf{X}) \ominus_{gH} (-1 \odot \textbf{Y}) \ominus_{gH} %\textbf{Z}= \ominus_{gH}(\textbf{X} \ominus_{gH} \textbf{Y} \ominus_{gH}(-1\odot %\textbf{Z}))
%\]			
For any $\textbf{W}, \textbf{Y}, \textbf{Z} \in I(\mathbb{R})$,
\[ \textbf 0\ominus_{gH} \{((-1 \odot \textbf{W}) \ominus_{gH} (-1 \odot \textbf{Y})) \ominus_{gH} (-1 \odot \textbf{Z})\}= ((\textbf{W} \ominus_{gH} \textbf{Y}) \ominus_{gH} \textbf{Z}).
\]			
\end{lemma}
\begin{proof}
See Appendix \ref{appendix_inq}.

\end{proof}
%%%%%%%%%%%%%%%%%%%%%%%%%%%%%%%%%%%%%%%%%%%%%%%%%%%%%%%%%%%%%%%%%%%%%%%%%%%%%%%%%%

\begin{lemma}\label{yure}
For all $\textbf{X},\textbf{Y},~\text{and}~\textbf{Z}~\text{of}~ I(\mathbb{R})$,
\begin{enumerate}[(i)]
    \item \label{3_1} if $\textbf 0 \preceq  \textbf{X} \ominus_{gH}\textbf{Y}$, then $\textbf 0 \ominus_{gH} \textbf{Z} \preceq (\textbf{X} \ominus_{gH} \textbf{Y})\ominus_{gH} \textbf{Z}$,
    \item \label{cv} if $\textbf{Z} \preceq  \textbf{X} \ominus_{gH} \textbf{Y}$, then $\textbf{Z} \ominus_{gH} \textbf{W} \preceq (\textbf{X}\ominus_{gH} \textbf{Y})\ominus_{gH} \textbf{W}$, for all $\textbf W \in I(\mathbb R)$,
   \item if $\textbf X \ominus_{gH} \textbf Y \preceq [L,L]$, then $[-L,-L] \preceq \textbf Y \ominus_{gH} \textbf X$, where $ L \in \mathbb{R}$\label{dbv},
   \item if $[-\gamma,-\gamma] \preceq \textbf X \ominus_{gH} \textbf Y$, then $\textbf Y \ominus_{gH} [\gamma,\gamma] \preceq \textbf X$, where $\gamma \in \mathbb{R}$\label{fds}, and
   \item if $\textbf Z \preceq \textbf X \oplus \textbf Y $, then $\textbf Z \ominus_{gH} \textbf Y \preceq \textbf X$\label{ewn}.
\end{enumerate}
 \end{lemma}

 \begin{proof}
 See Appendix \ref{appendix_ins}.

 \end{proof}
	
%\begin{definition}(\emph{Lower limit and $gH$-lower semicontinuity of extended IVFs} \cite{gourav2020}).
%The lower limit of an extended IVF $\textbf T: \mathcal Y \to I(\mathbb R)$ at $u \in \mathcal{Y}$, denoted as $\liminf\limits_{y \to u} \textbf T(y)$, is defined by
%\begin{align*}
 %          \liminf\limits_{y \to u}\textbf T(y) = &\lim\limits_{\alpha \to 0}(\inf\{\textbf T(y): y \in B_{\alpha}(u)\})\\
    %       =& \sup_{\alpha \geq 0}(\inf\{\textbf T(y): y \in B_{\alpha}(u)\}).
%\end{align*}
%$\textbf T$ is called $gH$-lower semicontinuous ($gH$-lsc) at a point $u \in \mathcal Y$ if
%\begin{align*}
 %       \textbf T(u) \preceq \liminf\limits_{y\to u} \textbf T(y).
%\end{align*}
%\end{definition}
%%%%%%%%%%%%%%%%%%%%%%%%%%%%%%%%%%%%%%%%%%%%%%%%%%%%%%%%%%%%%%%%%%%%%%%%%%%%%%%%
\begin{definition}{(\emph{Sequence in $I(\mathbb{R})^n$}}\cite{Ghosh2020lasso}). A function $\widehat{{\bf{\Phi}}}: \mathbb N \to I(\mathbb R)^n$ is called a sequence in $I(\mathbb R)^n$, where $\mathbb{N}$ is the set of natural numbers.
\end{definition}

%%%%%%%%%%%%%%%%%%%%%%%%%%%%%%%%%%%%%%%%%%%%%%%%%%%%%%%%%%%%%%%%%%%%%%%%%%%%%%%%
\begin{definition}{(\emph{Closed set in $I(\mathbb{R})^n$}}\cite{anshika}). \label{closedset}
A nonempty subset $\boldsymbol{\mathcal{U}} \subseteq I(\mathbb{R})^n$ is known to be \emph{closed} if for every convergent sequence $\{\widehat{\textbf{M}}_k\}$ in $\boldsymbol{\mathcal{U}}$ converging to $ \widehat{\textbf{M}}$, $ \widehat{\textbf{M}}$ must belong to $\boldsymbol{\mathcal U}$.
\end{definition}
%%%%%%%%%%%%%%%%%%%%%%%%%%%%%%%%%%%%%%%%%%%%%%%%%%%%%%%%%%%%%%%%%%%%%%%%%%%%%%%
\begin{definition}({\emph{Closure of a set in $I(\mathbb{R})^n$}}).
Let  ${\mathcal{Y}} \subseteq I(\mathbb{R})^n$. The intersection of all closed sets containing ${\mathcal{Y}}$ is called the closure of  ${\mathcal Y}$, abbreviated by ${cl}({\mathcal{Y}})$.
\end{definition}

\begin{definition}{(\emph{Convergent sequence in $I(\mathbb{R})^n$}\cite{Ghosh2020lasso}}). Let $\{ \widehat{\textbf{M}}_{k}\}$ be a sequence in $I(\mathbb{R})^n$. If there exists $\widehat{\textbf M} \in I(\mathbb{R})^n$ for which for any $\epsilon >0$ there exists $p \in \mathbb N$ such that
\begin{align*}
  \lVert \widehat{\textbf M}_{k} \ominus_{gH} \widehat{\textbf M} \rVert_{I(\mathbb{R})^n} <  \epsilon ~\text{for all}~ k \ge p, \end{align*}
   then $\{ \widehat{\textbf{M}}_{k}\}$ is said to be convergent and converges to $\widehat{\textbf M} $.
\end{definition}
\begin{remark} \label{uparmk}
It is to note that if a sequence $\{\widehat{\textbf M}_k\} =(\textbf M_{k1}, \textbf M_{k2}, \ldots, \textbf M_{kn})^\top$ in $I(\mathbb{R})^n$ converges to $\widehat{\textbf M} = (\textbf M_1, \textbf M_2, \ldots,  \textbf M_n)^\top \in I(\mathbb R)^n$, then by the definition of norm on $I(\mathbb R)^n$, the sequence ${\textbf M}_{{kj}}$ in $I(\mathbb{R})$ converges to $\textbf{M}_{j} \in I(\mathbb{R})$ for all  $j= 1,2,\ldots,n$. Also, according to the definition of norm on $I(\mathbb R)$, the sequences $\{\underline m_{kj}\}$ and $\{\overline m_{kj}\}$ in $\mathbb R$ converge to $\{\underline m_{j}\}$ and $\{\overline m_{j}\}$, respectively, for all $j$.
\end{remark}

%%%%%%%%%%%%%%%%%%%%%%%%%%%%%%%%%%%%%%%%%%%%%%%%%%%%%%%%%%%%%%%%%%%%%%%%%%%%%%%%%%

	\begin{definition}(\emph{Infimum and supremum of a subset of $\overline{I(\mathbb{R})}$} \label{lbound} \cite{gourav2020}). Let $\mathcal{U}\subseteq \overline{I(\mathbb{R})}$. We call an interval $\mathbf{{X}}\in I(\mathbb{R})$ a lower bound (respectively, an upper bound) of $\mathcal{U}$ if
	$\textbf{U}\in \mathcal{U} $ implies $
	\mathbf{{X}}\preceq \textbf{U}$ (respectively, $\mathbf{{U}} \preceq \textbf{X}$). \\
A lower bound $\mathbf{{X}}$ of $\mathcal{U}$ is called \emph{infimum} of $\mathcal{U}$, denoted by $\inf\mathcal{U}$, if for any lower bound $\textbf Z$ of $\mathcal U$, $\textbf{Z}\preceq \mathbf{{X}}$. \\
An upper bound $\mathbf{{X}}$ of $\mathcal{U}$ is called \emph{supremum} of $\mathcal{U}$, denoted by $\sup\mathcal{U}$, if for any upper bound $\textbf{Z}$  of $\mathcal{U}$, $\mathbf{{X}}\preceq \textbf{Z}$.
\end{definition}

%%%%%%%%%%%%%%%%%%%%%%%%%%%%%%%%%%%%%%%%%%%%%%%%%%%%%%%%%%%%%%%%%%%%%%%%%%%%%%%%%%
% 	\begin{definition}(\emph{Supremum of a subset of $\overline{I(\mathbb{R})}$} \label{ubound}\cite{gourav2020}). Let $\mathcal{U}\subseteq \overline{I(\mathbb{R})}$. An interval $\mathbf{{X}}\in I(\mathbb{R})$ is  an upper bound of $\mathcal{U}$ if
% 	$\textbf{Y} \in \mathcal{U}$ implies $
% 	\textbf{Y}\preceq \mathbf{{X}}$.
% 	\end{definition}
%%%%%%%%%%%%%%%%%%%%%%%%%%%%%%%%%%%%%%%%%%%%%%%%%%%%%%%%%%%%%%%%%%%%%%%%%%%%%%%%%%

	\begin{remark} \cite{gourav2020}
		Let $\mathcal{S}=\left\{[a_\mu,b_\mu]\in \overline{I(\mathbb{R})}: \mu\in\Lambda~ \text{and}~ \Lambda~\text{being an index set}\right\}$. Then, by Definition \ref{lbound}, it follows that $ \footnotesize
		{\inf\mathcal{S}=\left[\inf\limits_{\mu\in\Lambda}a_\mu,~\inf\limits_{\mu\in\Lambda}b_\mu\right] \text{and}
		 \sup\mathcal{S}=\left[\sup\limits_{\mu\in\Lambda}a_\mu,~\sup\limits_{\mu\in\Lambda}b_\mu\right]}.$
	\end{remark}

%%%%%%%%%%%%%%%%%%%%%%%%%%%%%%%%%%%%%%%%%%%%%%%%%%%%%%%%%%%%%%%%%%%%%%%%%%%%%%%%%	
%\begin{definition}\label{inf}(\emph{Infimum for IVF} \cite{gourav2020}). Let $\mathcal{S}$ be a nonempty subset of $\mathcal Y$ and $\textbf{T}: \mathcal{S} \rightarrow \overline{I(\mathbb{R})}$ be an extended IVF. $\textbf{T}$ is bounded below on $\mathcal S$ if there exists an interval $\textbf M=[\underline m, \overline m]$ with \[\textbf M \preceq \textbf T(y) ~\text{for all}~ y \in \mathcal S.\]
%Then, the infimum of $\textbf{T}$ over $\textbf S$ denoted as    $ \underset{y \in \mathcal{S}}{\inf}\textbf{T}$ is equal to infimum of the range set of $\textbf T,$ i.e.,
%	\[\inf_{y \in \mathcal{S}}\textbf{T}=\inf\{\textbf{T}(y): y\in \mathcal{S}\}.
%	\]
%%%%%%%%%%%%%%%%%%%%%%%%%%%%%%%%%%%%%%%%%%%%%%%%%%%%%%%%%%%%%%%%%%%%%%%%%%%%%%%%%%
%\begin{dfn}\label{sup}(\emph{Supremum for IVF} \cite{gourav2020}). Let $\textbf{S}$ be a nonempty subset of $\mathcal Y$ and $\textbf{T}: \textbf{S} \rightarrow \overline{I(\mathbb{R})}$ be an extended IVF.
%if there exists a interval $M=[\underline m, \overline m]$ such that \[ \textbf T(y) \preceq M
%~\text{for all}~ y \in \textbf S.\]
%Similarly, the supremum of $\textbf{T}$ over $\mathcal S$ is defined by
%	\[\sup_{y \in \mathcal{S}}\textbf{T}=\sup\{\textbf{T}(y): y\in \mathcal{S}\}.
%	\]
%	\end{definition}

	%

	%

%%%%%%%%%%%%%%%%%%%%%%%%%%%%%%%%%%%%%%%%	

\section{\textbf{$gH$-weak subdifferential calculus for IVFs}}\label{section3}	
In this section, we introduce the ideas of $gH$-weak subgradient and $gH$-weak subdifferential for IVFs. Some properties of $gH$-weak subdifferential and an inclusion for sum rule are provided. Its relation with $gH$-Fr\'echet lower subdifferential is also discussed.

\begin{definition}\label{etuhj}(\emph{$gH$-weak subdifferential}).
Let $\emptyset \not= \mathcal{Y} \subseteq \mathbb{R}^{n}$ and ${\bf{\Phi}}$ be an IVF defined on $\mathcal{Y}$. A pair $(\widehat{\textbf{G}^w}, c) \in I(\mathbb{R})^n \times \mathbb{R}_{+}$ is said to be a $gH$-weak subgradient of ${\bf{\Phi}}$ at $u \in \mathcal{Y}$ if for every $y \in \mathcal{Y}$,
\begin{align} \label{mnbvc}
 \widehat{\textbf{G}^{w}}^{\top} \odot (y-u) \ominus_{gH}  c \lVert  y-u \rVert\preceq{\bf{\Phi}}(y) \ominus_{gH} {\bf{\Phi}}(u).
\end{align}
The set of all $gH$-weak subgradients of ${\bf{\Phi}}$ at $ u \in \mathcal{Y}$, i.e.,
\[ \partial^{w} {\bf{\Phi}}(u)=\bigg\{(\widehat{\textbf{G}^{w}}, c) \in I(\mathbb{R})^n \times \mathbb{R}_{+}:\widehat{\textbf{G}^{w}}^{\top} \odot (y-u) \ominus_{gH}  c \lVert  y-u \rVert\preceq{\bf{\Phi}}(y) \ominus_{gH} {\bf{\Phi}}(u)~ \forall~ y \in \mathcal{Y}\bigg\}
\]
is said to be $gH$-weak subdifferential of ${\bf{\Phi}}$ at $u \in \mathcal{Y}$.

%where $\textbf{A}=(\textbf{A}_1,\textbf{A}_2,\cdots,\textbf{A}_n)$ and $\textbf{B}=(\textbf{B}_1,\textbf{B}_2,\cdots,\textbf{B}_n)$.

    \end{definition}

\begin{example}	
\label{ret}
Let an IVF ${\bf{\Phi}}: [-1,1] \rightarrow I(\mathbb{R})$ be defined by
\[ {\bf{\Phi}}(y) = \left[y^2, \lvert y \rvert\right],
 \text{where} \  y \in [-1,1].\]
\end{example}
Let us compute the $gH$-weak subdifferential  of ${\bf{\Phi}}$ at $0$  and $1$, i.e., $\partial^{w}{\bf{\Phi}}(0)$ and $ \partial^{w}{\bf{\Phi}}(1)$, respectively.
Note that
\begin{eqnarray*}
\partial^{w}{{\bf{\Phi}}}(0)&=&\left\{(\textbf{G}^{w}_{1}, c) \in I(\mathbb{R}) \times \mathbb{R}_{+}: \textbf{G}^{w}_{1} \odot y \ominus_{gH} c\lvert y \rvert\preceq \left[y^2, \lvert y \rvert\right]~ \forall~ y \in [-1,1]\right\}\\
&=&  \left\{\left([\underline {g^{w}_{1}},\overline {{g^{w}_{1}}}], c\right) \in I(\mathbb{R}) \times \mathbb{R}_{+}: [\underline {g^{w}_{1}}, \overline {g^{w}_{1}}] \odot y \ominus_{gH} c \lvert y \rvert \preceq \left[y^2,\lvert y\rvert\right] ~\forall~ y \in [-1,1]\right\},
\end{eqnarray*}
which yields the following two cases corresponding to $y \in [0,1]$ and $y \in [-1,0]$.
\begin{enumerate}[$\bullet$ \textbf{Case} 1.]
\item \label{hgdxgg}
\begin{eqnarray*}
\partial^{w}{\bf{\Phi}}(0)
&=&  \bigg\{\left([\underline {g^{w}_{1}},\overline {g^{w}_{1}}], c\right) \in I(\mathbb{R}) \times \mathbb{R}_{+}: [\underline {g^{w}_{1}}, \overline {g^{w}_{1}}] \odot y \ominus_{gH} c \lvert y \rvert \preceq \left[y^2,\lvert y\rvert\right]~ \forall~ y \in [0,1]\bigg\}\\
&=& \bigg\{\left([\underline {g^{w}_{1}},\overline {g^{w}_{1}}], c\right) \in I(\mathbb{R}) \times \mathbb{R}_{+}:  \underline {g^{w}_{1}}y-cy \leq y^2~\text{and}~  \overline {g^{w}_{1}}y-cy\leq y~ \forall~ y \in [0,1]\bigg\}\\
&=& \bigg\{\left([\underline {g^{w}_{1}},\overline {g^{w}_{1}}], c\right) \in I(\mathbb{R}) \times \mathbb{R}_{+}: \underline {g^{w}_{1}}-c \leq0 ~\text{and}~ \overline {g^{w}_{1}} -c \leq 1\bigg\}.
\end{eqnarray*}
\item \label{khdcgg} Likewise,
\[\partial^{w}{\bf{\Phi}}(0)=\bigg\{\left([\underline {g^{w}_{1}},\overline {g^{w}_{1}}], c\right) \in I(\mathbb{R}) \times \mathbb{R}_{+}:-1 \leq\underline {g^{w}_{1}}+c  ~\text{and}~~ 0 \leq\overline {g^{w}_{1}} +c \bigg\}.\]
\end{enumerate}
Hence, by combining \textbf{Case} \ref{hgdxgg} and \textbf{Case} \ref{khdcgg}, we obtain
\[\partial^w{\bf {\Phi}}(0)=\bigg\{({\textbf G}^w_{1}, c)\in I(\mathbb{R}) \times \mathbb{R}_{+}: [-1-c,-c]\preceq {\textbf {G}}^w_{1} \preceq [c, 1+c]\bigg\}.\]
Similarly,   \[\partial^w{\bf {\Phi}}(1)=\bigg\{({\textbf G}^w_{2}, c )\in I(\mathbb{R}) \times \mathbb{R}_{+}: [1-c,2-c]\preceq {\textbf {G}}^w_{2}\bigg\}.\]
%%%%%%%%%%%%%%%%%%%%%%%%%%%%%%%%%%%%%
\begin{remark}
To understand the geometric interpretation of the $gH$-weak subdifferential of an IVF ${\bf{\Phi}}$, let $ (\widehat{\textbf G^{w}}, c) \in \partial^{w} {\bf{\Phi}}(u)$.  This means that $(\widehat{\textbf G^w}, c) \in I(\mathbb R)^n \times \mathbb R_{+}$, for every $c \ge 0$, is a $gH$-weak subgradient of ${\bf{\Phi}}$ at $u \in \mathcal Y$ if and only if there exists a  concave and $gH$-continuous IVF $\textbf H : \mathcal Y \to I(\mathbb R)$, which is defined  by
\[
     \textbf H(y) = {\bf{\Phi}}(u) \oplus \widehat{\textbf G^w}^\top \odot (y-u) \ominus_{gH} c \lVert y-u \rVert ~\forall~ y \in \mathcal Y,\]
that satisfies
\[  (\forall~ y \in \mathcal Y) ~\textbf H(y) \preceq {{\bf{\Phi}}}(y) ~\text{and}~ \textbf H(u) = {\bf{\Phi}}(u).\]
This condition shows that $\textbf H$ must intersect ${\bf{\Phi}}$ at least at the point $(u, {\bf{\Phi}}(u))$ from bottom. Hence, it concludes that if ${\bf{\Phi}}$ is $gH$-weak subdifferentiable at $u$ and $ (\widehat{\textbf G^w}, c) \in \partial^{w} \boldsymbol {\Phi}(u) $, then the graph of IVF $\textbf H$, that is,
\begin{align*}
     Gr (\textbf H) = \{ (y, \textbf Y) \in \mathcal Y \times I(\mathbb R) : \textbf Y= \textbf H(y)\}
\end{align*}
always lie below the  epigraph of ${\bf{\Phi}}$, i.e.,
$$ \text{Epi} ({\bf{
\Phi}}) = \{(y, \textbf Y)\in  \mathcal Y \times I(\mathbb R): {{\bf{\Phi}}}(y)\preceq \textbf Y  \}
   ,$$
such that
\begin{align*}
\text{Epi}({\bf{\Phi}}) \subset \text{Epi}(\textbf H)~ \text{and}~ cl(\text{Epi} ({\bf{\Phi}})) \bigcap  Gr (\textbf H) ~\text{is nonempty}.~
\end{align*}
\end{remark}
\renewcommand{\figurename}{Figure}
\begin{figure}[h]
\centering
\includegraphics[width=10cm, height=7cm]{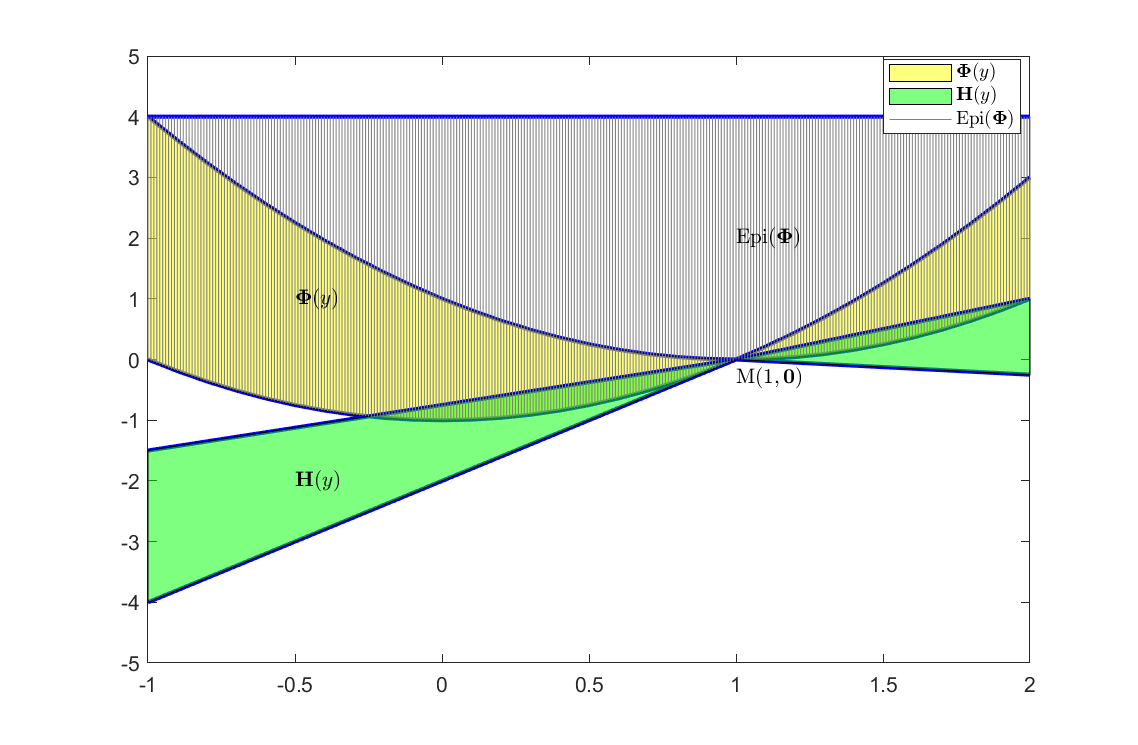}
\caption{The geometrical  view of the graph of IVF ${\bf{\Phi}}$ (yellow) and  the graph of IVF $\textbf H$ (green), which intersects the gray shaded region ($\text{Epi}({\bf{\Phi}})$) from below.} \label{figure1}
\end{figure}
%%%%%%%%%%%%%%%%%%%

%%%%%%%%%%%%%%%%%%%
 For example, Let $\mathcal Y = [-1,2]$. Consider
 an IVF ${\bf{\Phi}}: \mathcal Y \to I(\mathbb R)$ which is given by
$$
 {{\bf{\Phi}}}(y)=
 \begin{cases}
 \left[ y^2-1, (y-1)^2\right] & ~\text{if}~ y \in [-1,1]\\
  \left[ (y-1)^2, y^2-1\right] & ~\text{if}~ y \in (1,2].
 \end{cases}$$
The $gH$-weak subdifferential of ${\bf{\Phi}}$ at $u = 1$ is
\begin{align*}
 \partial^w{\bf{\Phi}}(1)=\left\{({\textbf G}^{w}, c )\in I(\mathbb{R}) \times \mathbb{R}_{+}: [-c,2-c]\preceq {\textbf {G}^w} \preceq [c, 2+c]\right\}.
\end{align*}
For instance, $(\textbf G^w, c)= ([0.25, 1.5], 0.5) \in \partial^{w}{\bf{\Phi}}(1)$, geometrically indicates that the IVF $$ \textbf H(y)=
{\bf{\Phi}}(1)\oplus [0.25,1.5] \odot (y-1) \ominus_{gH} 0.5 \lvert y-1\rvert$$  intersects
 $$
 \text{Epi}({\bf{\Phi}}) = \{ (y,\textbf 4 )\in \mathcal Y \times \mathbb R: {{\bf{\Phi}}}(y)\preceq \textbf 4\}$$
at the point M$(1,\textbf 0)$ from below as shown in Figure \ref{figure1}.
We also observe from the figure that
\begin{align*}
   \text{ Epi}({\bf{\Phi}}) \subset \text{Epi}(\textbf H),~ \text{and}~ cl(\text{Epi}({\bf{\Phi}}))\bigcap  Gr(\textbf H)~\text{is nonempty}. \\
\end{align*}

\begin{theorem}{\emph{(Convexity of $gH$-weak subdifferential).}}\label{opi}
Let $\mathcal{Y} \subset \mathbb{R}^n$. Let the $gH$-weak subdifferential of $\bf{\Phi}: \mathcal{Y} \rightarrow I(\mathbb{R})$ at $u$ be nonempty. Then, the set $\partial ^w{\bf{\Phi}}(u)$ is convex.
\end{theorem}
	
	\begin{proof}
	
	Let $(\widehat{\textbf G^{w}_{1}},c_{1})$ and $(\widehat{\textbf G^{w}_{2}}, c_{2}) \in  \partial^{w}{\bf{\Phi}}(u)$, where $\widehat{\textbf G^{w}_1}=(\textbf G^w_{11}, \textbf G^w_{12}, \ldots, \textbf G^w_{1n})^\top,  \widehat{\textbf G^{w}_2}=(\textbf G^w_{21}, \textbf G^w_{22}, \ldots, \textbf G^w_{2n})^\top$. Let $\beta \in [0,1]$. From the definition of $\partial^{w}{\bf{\Phi}}(u)$, we have
 \begin{align}\label{ghu}
&    \widehat{\textbf{G}^{w}_{1}}^{\top} \odot (y-u) \ominus_{gH} c_{1}\lVert y-u \rVert \preceq {{\bf{\Phi}}}(y)\ominus_{gH} {\bf{\Phi}}(u) ~\text{and}~\\
& \label{ppoiuytew}   \widehat{\textbf{G}^{w}_{2}}^{\top} \odot (y-u) \ominus_{gH} c_{2}\lVert y-u \rVert \preceq {{\bf{\Phi}}}(y)\ominus_{gH} {\bf{\Phi}}(u),
 \end{align}
for all $y \in \mathcal{Y}$. Up to a rearrangement of terms, let the first $m$  components of $(y-u)$ be non-negative and the rest be negative. Then, from the inequalities (\ref{ghu}) and (\ref{ppoiuytew}), we get
\begin{align}
 & \bigoplus_{i=1}^{m} ~(y_{i}-u_{i}) \odot \textbf{G}^w_{1i} \bigoplus_{j=m
 +1}^{n} (y_{j}-u_{j}) \odot \textbf{G}^w_{1j} \ominus_{gH} c_{1}\lVert y-u \rVert\preceq {\bf{\Phi}}(y) \ominus_{gH} {\bf{\Phi}}(u) ~\nonumber
 \end{align}
 and
 \begin{align}
 &  \bigoplus_{i=1}^{m} ~(y_{i}-u_{i}) \odot \textbf{G}^w_{2i} \bigoplus_{j=m
 +1}^{n} (y_{j}-u_{j}) \odot \textbf{G}^w_{2j} \ominus_{gH} c_{2}\lVert y-u \rVert \preceq{{\bf{\Phi}}}(y) \ominus_{gH} {\bf{\Phi}}(u).  \nonumber
 \end{align}
 Thus,
 \begin{align}
 &\label{gjk}
  \bigoplus_{i=1}^{m}~ \beta \odot ((y_{i}-u_{i}) \odot \textbf{G}^w_{1i}) \bigoplus_{j=m+1}^{n} \beta \odot ((y_{j}-u_{j}) \odot \textbf{G}^w_{1j})\ominus_{gH} \beta c_{1} \lVert y-u \rVert  \preceq   \beta \odot ({{\bf{\Phi}}}(y)\ominus {\bf{\Phi}}(u)) \end{align}
  {and}
  \begin{align}\label{hui}
&\bigoplus_{i=1}^{m} ~(1-\beta) \odot ((y_{i}-u_{i}) \odot \textbf{G}^w_{2i}) \bigoplus_{j=m+1}^{n} (1-\beta) \odot ((y_{j}-u_{j}) \odot \textbf{G}^w_{2j})
 \ominus_{gH} (1-\beta) c_{2} \lVert y-u \rVert \nonumber\\
 \preceq  &~ (1-\beta) \odot ({{\bf{\Phi}}}(y)\ominus {\bf{\Phi}}(u)).
 \end{align}
 By adding  (\ref{gjk}) and (\ref{hui}), we obtain
 \begin{align}
  & \bigoplus_{i=1}^{m} ~(y_{i}-u_{i})\odot \{\beta \odot \textbf{G}^{w}_{1i}\oplus (1-\beta)\odot \textbf{G}^{w}_{2i}\} \bigoplus_{j=m+1}^{n}(y_{j}-u_{j})\odot \{\beta \odot \textbf{G}^{w}_{1j}\oplus
  (1-\beta) \odot \textbf{G}^{w}_{2j}\} \nonumber\\
  & \ominus_{gH} (\beta  c_{1}
  \oplus (1-\beta)  c_{2}) \lVert y-u \rVert \preceq {{\bf{\Phi}}}(y) \ominus_{gH} {\bf{\Phi}}(u).
  \end{align}
  Therefore, we have
  \[  \{\beta \odot \widehat{\textbf{G}^{w}_{1}}\oplus (1-\beta)\odot \widehat{\textbf{G}^{w}_{2}}\}^{\top} \odot(y-u)  \ominus_{gH} (\beta c_{1}
  \oplus (1-\beta) c_{2}) \lVert y-u \rVert  \preceq    {{\bf{\Phi}}}(y) \ominus_{gH} {\bf{\Phi}}(u),\]
  i.e., $(\beta\odot \widehat{\textbf{G}^{w}_{1}}\oplus (1-\beta) \odot \widehat{\textbf{G}^{w}_{2}} , ~\beta c_{1} \oplus (1-\beta )c_{2}) \in \partial^w {\bf{\Phi}}(u)$, which proves the convexity of $\partial^{w}{\bf{\Phi}}(u)$.
 	\end{proof}

\begin{theorem}{\emph{(Closedness of $gH$-weak subdifferential).}} \label{opw}
Let $\emptyset \neq \mathcal{Y} \subseteq I(\mathbb{R})^n$. If for an IVF $\bf{\Psi}: \mathcal{Y} \rightarrow I(\mathbb{R})$, the set $\partial ^{w}{\bf{\Psi}}(u)$ is nonempty at $u \in \mathcal{Y}$, then $\partial ^{w}{\bf{\Psi}}(u)$ is closed.
\end{theorem}

\begin{proof}
Let $\{(\widehat{\textbf{G}^{w}_{k}}, c_k )\}$  be an arbitrary sequence in $\partial^{w}{{\bf{\Psi}}}(y)$  converging to $(\widehat{\textbf G^w}, c) \in I(\mathbb{R})^n \times \mathbb{R}_{+}$, where  $\widehat{\textbf G^{w}_k}= (\textbf G^w_{k1},\textbf G^w_{k2},\ldots,\textbf G^w_{kn})^\top$ and $\widehat{\textbf G^w}=(\textbf G^w_{1},\textbf G^w_{2},\ldots,\textbf G^w_{n})^\top$. Since ($\widehat{\textbf G^{w}_k}, c) \in \partial^{w}\textbf{T}(y)$ for all $d \in \mathcal{Y}$, we obtain
\[  \widehat{\textbf G^{w}_k}^\top \odot d\ominus_{gH} c_{k} \lVert d \rVert \preceq{\bf{\Psi}}(u+ d) \ominus_{gH} {\bf{\Psi}}(u),\]
which implies
\begin{align} \label{bnmlkj}
\bigoplus_{i=1}^{n} d_{i} \odot \textbf{G}^{w}_{ki} \ominus_{gH} c_{k} \lVert d \rVert  \preceq{\bf{\Psi}}(u+d) \ominus_{gH} {\bf{\Psi}}(u).
\end{align}
Up to a rearrangement of terms, let the first $p$ components  of $d$ be non-negative and the rest be negative.
Then, from (\ref{bnmlkj}), we get
\[   \bigoplus_{i=1}^p
 ~d_{i} \odot \textbf{G}^{w}_{ki} \bigoplus_{j=p+1}^{n} d_{j} \odot \textbf{G}^{w}_{kj} \ominus_{gH} c_{k} \lVert d \rVert \preceq{\bf{\Psi}}(u + d) \ominus_{gH} {\bf{\Psi}}(u)  \]
 \[\implies  \bigoplus_{i=1}^p
 d_{i} \odot [\underline {g^{w}_{ki}}, \overline {g^{w}_{ki}}] \bigoplus_{j=p+1}^{n} d_{j} \odot [\underline {g^{w}_{kj}}, \overline {g^{w}_{kj}}] \ominus_{gH} c_{k} \lVert d \rVert\preceq{\bf{\Psi}}(u+ d) \ominus_{gH} {\bf{\Psi}}(u).\]
 Therefore,
 \begin{align}
 & \label{gfd}  \sum_{i=1}^{p} \underline {g^{w}_{ki}} d_{i} + \sum_{j=p+1}^{n} \overline {g^{w}_{kj}}d_{j}-c_{k}\lVert d \rVert\preceq \min\left\{\underline \Psi(u+d)-\underline \Psi(u), \overline \Psi(u+d)-\overline \Psi(u) \right\}
 \end{align}
 $\text{and}$
 \begin{align}
 & \label{sde}  \sum_{i=1}^{p}  \overline {g^{w}_{ki}} d_{i} + \sum_{j=p+1}^{n} \underline {g^{w}_{kj}} d_{j} - c_{k} \lVert d \rVert \preceq \max\left\{\underline \Psi(u+d)-\underline \Psi(u), \overline \Psi(u+d) -\overline \Psi(u) \right\}.
 \end{align}
 Since the sequence $\widehat{\textbf G^{w}_{k}}$ converges to $\widehat{\textbf G^w}$, the sequences $\{\underline {g^{w}_{ki}}\}$ and$\{\overline {g^{w}_{ki}}\}$ converge to $\{\underline {g^{w}_{i}}\}$ and $\{\overline {g^{w}_{i}}\}$, respectively for all $i$. Thus, by  (\ref{gfd}) and (\ref{sde}), we have
 \begin{align*}
 \sum_{i=1}^{p} \underline {g^{w}_{ki}} d_{i} + \sum_{j=p+1}^{n} \overline {g^{w}_{kj}} d_{j}- c_{k} \lVert d \rVert &\rightarrow  \sum_{i=1}^{p} \underline {g^{w}_{i}} d_{i} + \sum_{j=p+1}^{n} \overline {g^{w}_{j}} d_{j}- c \lVert d \rVert \\
&\preceq \min\bigg\{ \underline {\Psi}(u+ d)-\underline \Psi(u),\overline {\Psi}(u+ d)-\overline \Psi(u)\bigg\}
\end{align*}
$\text{and}$
\begin{align*}
\sum_{i=1}^{p} \overline {g^{w}_{ki}} d_{i} + \sum_{j=p+1}^{n} \underline {g^{w}_{kj}} d_{j}- c_{k} \lVert d \rVert& \rightarrow \sum_{i=1}^{p} \overline {g^{w}_{i}} d_{i} + \sum_{j=p+1}^{n} \underline {g^{w}_{j}} d_{j}- c \lVert d \rVert \\ &\preceq \max\bigg\{ \underline \Psi(u+ d)-\underline \Psi(u),\overline \Psi(u+ d)-\overline \Psi(u)\bigg\}.
\end{align*}
Hence, for any $u  \in \mathcal{Y}$,
\begin{align*}
&  \bigg[ \sum_{i=1}^{p} \underline {g^{w}_{i}} d_{i}+ \sum_{j=p+1}^n\overline {g^{w}_{j}} d_{j} -c\lVert d \rVert,\sum_{i=1}^{p}\overline {g^{w}_{i}} d_{i}+ \sum_{j=p+1}^{n}\underline {g^{w}_{j}} d_{j} -c\lVert d \rVert\bigg]\preceq{\bf{\Psi}}(u+ d) \ominus_{gH} {\bf{\Psi}}(u) \\
\implies&  \bigoplus_{i=1}^{p}~[\underline{ g^{w}_{i}}d_{i},\overline {g^{w}_{i}}d_{i} ] \bigoplus_{j=p+1}^{n}[\overline {g^{w}_{j}}d_{j},\underline {g^{w}_{j}}d_{j} ]\ominus_{gH} c\lVert d \rVert\preceq{\bf{\Psi}}(u+ d) \ominus_{gH} {\bf{\Psi}}(u) \\
\implies&\bigoplus_{i=1}^{p} ~d_{i} \odot {\textbf G}^{w}_{i} \bigoplus_{j=p+1}^{n} d_{j} \odot {\textbf{G}^{w}_{j}} \ominus_{gH} c\lVert d \rVert \preceq {\bf{\Psi}}(u+d) \ominus_{gH} {\bf{\Psi}}(u)\\
\implies&    \widehat{\textbf G^w}^{\top} \odot d\ominus_{gH} c\lVert d \rVert \preceq{\bf{\Psi}}(u+ d) \ominus_{gH} {\bf{\Psi}}(u).
 \end{align*}
 Therefore, $\widehat{\textbf {G}^{w}} \in \partial^{w} {\bf{\Psi}}(u)$, and hence $\partial^{w} {\bf{\Psi}}(u)$ is closed.
\end{proof}

\begin{definition}\label{werrr_1}{($gH$-\emph{Fr\'{e}chet lower subdifferential})}.
Let $\bf{\Phi}: \mathcal{Y} \rightarrow I(\mathbb{R}) \cup \{-\infty,+\infty\}$ be an IVF that is finite at an $u \in \mathcal{Y}$. Then, the $gH$-Fr\'{e}chet lower subdifferential of $\bf{\Phi}$ at $u$ is defined by
\begin{align*}
\partial^{-}_{\mathscr{F}} {{\bf{\Phi}}}(u)=\bigg\{\widehat{\textbf{G}}:~&\textbf 0 \preceq {\liminf_{\substack{%
y \to u\\ y \neq u}} \frac{1}{\lVert y-u \rVert}\odot \{{{\bf{\Phi}}}(y) \ominus_{gH} {\bf{\Phi}}(u) \ominus_{gH} \widehat{\textbf G}^{\top}\odot(y-u)}\},\\
&\text{where} ~\widehat{\textbf G}: \mathcal{Y} \rightarrow I(\mathbb{R})~ \text{is $gH$-continuous and linear IVF}~ \bigg\}.
\end{align*}
\end{definition}
One important fact is that $gH$-weak subdifferential is an  immediate consequence of $gH$-Fr\'{e}chet lower subdifferential.
 \begin{theorem} \label{cvb1}
Let $\emptyset \neq \mathcal{Y} \subseteq \mathbb{R}^n$. If $\bf{\Phi}: \mathcal{Y} \to I(\mathbb{R})$ has $gH$-Fr\'{e}chet lower subdifferential $\widehat{\textbf G}$ at the point $u$, then $(\widehat{\textbf G}, \epsilon)$ is a $gH$-weak subgradient of $\bf{\Phi}$  at $u$ for any $\epsilon \in \mathbb{R}_{+}$.
 \end{theorem}
 \begin{proof}
Let $\widehat{\textbf G} \in  \partial^{-}_{\mathscr{F}
}{{\bf{\Phi}}}(u)$. Due to Definition \ref{werrr_1}, we can write
\[ \textbf 0 \preceq \liminf_{\substack{%
y \to u\\ y \neq u}} \frac{1}{\lVert y-u \rVert} \odot \{{{{\bf{\Phi}}}(y) \ominus_{gH} {\bf{\Phi}}(u) \ominus_{gH} \widehat{\textbf G}^{\top}\odot(y-u)}\}.
\]
Then, for the $\epsilon > 0$ in the hypothesis there exists $\delta > 0$ such that
 \[-\epsilon\lVert y-u\rVert \preceq {{\bf{\Phi}}}(y) \ominus_{gH} {\bf{\Phi}}(u) \ominus_{gH} \widehat{\textbf{G}}^{\top} \odot (y-u) ~ \forall ~y \in  B_{\delta}(u),
 \]
Then, from Lemma \ref{dr2}, we have  \[ \widehat{\textbf  G}^{\top}\odot (y-u)  \ominus_{gH}  \epsilon \lVert y-u \rVert \preceq {{\bf{\Phi}}}(y) \ominus_{gH} {\bf{\Phi}}(u).
 \]
 By Definition \ref{etuhj}, $(\widehat{\textbf G}, \epsilon)$ is a $gH$-weak subdifferential of $\textbf{T}$
 at $u$.
 \end{proof}

  \begin{lemma}\label{yuv}
For any $y \in \mathbb{R}^n$ and $\widehat{\textbf C}=(\textbf C_{1}, \textbf C_{2}, \textbf C_{3}, \ldots, \textbf C_{n}) \in I(\mathbb{R})^n$,
\[ - \lVert y \rVert  \lVert \widehat{\textbf C}\rVert_{I(\mathbb{R})^n}  \preceq  \lVert y^{\top} \odot \widehat{\textbf C} \rVert_{I(\mathbb{R})}
\]
\end{lemma}
\begin{proof}  See Appendix \ref{appendix_iou}.
\end{proof}
To investigate the class of interval-valued functions for which weak subgradients always exist, we need the following definition.

\begin{definition} (\emph{$gH$-lower Lipschitz IVF}).
Let $\emptyset \neq \mathcal{Y} \subseteq \mathbb{R}^n$. An IVF ${\bf{\Phi}} : \mathcal Y \rightarrow \overline{I({\mathbb{R}})}$ is called $gH$-lower locally Lipschitz at $u \in \mathcal Y$ if $ \exists~ L\geq 0$  and a neighbourhood $\mathcal N(u)$ of $u$ such that
\begin{align} \label{sdfhh}
     -  L \lVert y- u \rVert\preceq {{\bf{\Phi}}}(y) \ominus_{gH} {\bf{\Phi}}(u) ~\forall ~y \in \mathcal N(u).
\end{align}
If the inequality (\ref{sdfhh})  satisfies for all $y \in \mathcal Y$, then ${\bf{\Phi}} $ is called $gH$-lower Lipschitz at $u \in \mathcal Y$ with Lipschitz constant $ L$.
\end{definition}

\begin{example}
 Let ${\bf{\Phi}}: [1, \infty) \to I(\mathbb R)$ be an IVF, defined by ${\bf{\Phi}}(y)= \ln{y} \odot \textbf C$ for all $y \in [1,\infty)$, where $\textbf 0 \preceq \textbf C = [\underline c, \overline c]$. Let $\delta >0$. We choose the neighbourhood of $u$,  $\mathcal N_{\delta}(u) = \{y : \lvert y-u \rvert < \delta  \}$.\\
If $0 < y-u< \delta$, then $u < y$ and also then $\frac{y}{u}>1$ and then
\begin{align} \label{logequ_1}
    0 < \ln \frac{y}{u} &< \frac{y}{u}-1,~ \text{since}~ \ln (1+p) < p ~\text{if}~ p >0 \nonumber \\
    &  \leq y - u.
\end{align}
Since $ \underline c, \overline c \geq 0$,  we have
 $$ (\ln y -\ln u) \underline c \leq (y-u) \underline c ~\text{and}~ (\ln y -\ln u) \overline c \leq (y-u) \overline c.  $$ Then,
 \begin{align}\label{inmulog_1}
   (\ln y -\ln u) \odot \textbf C \preceq (y-u) \odot \textbf C .
   \end{align}
If $-\delta < y-u <0$, then $y < u$ and also then $ \frac{u}{y} >1$ and then
\begin{align}\label{logequ_2}
    0 < \ln \frac{u}{y} & <\frac{u}{y}-1, ~\text{since}~ \ln (1+p) < p~ \text{if}~ p >0 \nonumber \\
    & \leq u- y.
\end{align}
Then, similarly, as seen in (\ref{inmulog_1}),
\begin{align}\label{inmulog_2}
 (\ln u -\ln  y) \odot \textbf C \preceq (u- y) \odot \textbf C .
\end{align}
Combining ($\ref{inmulog_1}$) and ($\ref{inmulog_2}$), we have
 \begin{align*}
   &  \lvert \ln y -\ln u \rvert \odot \textbf C \preceq {\lvert y-u \rvert} \odot \textbf C\\
  \implies &  \ln u \odot \textbf C \ominus_{gH} \ln  y\odot \textbf C  \preceq   \lvert y-u \rvert  \odot \textbf C \\
  \implies & - \lvert y-u \rvert  \odot \textbf C \preceq \ln y \odot \textbf C \ominus_{gH} \ln u \odot \textbf C \\
  \implies & - \overline c \lvert y -u \rvert \preceq {\bf{\Phi}}(y) \ominus_{gH} {\bf{\Phi}}(u).
 \end{align*}
This shows that ${\bf{\Phi}}$ is $gH$-lower locally Lipschitz on $\mathcal N_{\delta}(u) $ with  $L =\overline c$. From arbitrariness of $y, u$ in $[1, \infty)$, we conclude that ${\bf{\Phi}}$ is $gH$-lower Lipschitz on $[1, \infty)$.
\end{example}

 \begin{theorem}\label{alg_1}
 \label{hgfd}
Let $\emptyset \neq \mathcal Y \subseteq \mathbb{R}^n$. %$ u $ is a given point of $ \mathcal Y$.
Let ${\bf{\Phi}} : \mathcal Y \rightarrow \overline {I(\mathbb{R})}$ be an IVF, where ${\bf{\Phi}}(u) $ is finite for some $ u \in  \mathcal Y$. Then, the following three statements are equivalent:
\begin{enumerate}[(a)]
\item ${\bf{\Phi}}$ is $gH$-weak subdifferentiable at $ u$.
\item ${\bf{\Phi}}$ is $gH$-lower Lipschitz at $u$.
\item ${\bf{\Phi}}$ is $gH$-lower locally Lipschitz at $u$, and there exists a number $p \geq 0$ and an interval $\textbf Q$ such that
\begin{align} \label{ewr}
 -p \lVert y \rVert \oplus \textbf Q\preceq{\bf{\Phi}}( y)~ \forall~ y \in \mathcal Y.
\end{align}
\end{enumerate}
 \end{theorem}
 \begin{proof}
 \fbox{(a) implies (b)}~: Suppose ${\bf{\Phi}} $ is $gH$-weak subdifferentiable at $u$. Then, there exists $(\widehat{\textbf G^w},c) \in I(\mathbb{R})^n \times \mathbb{R}_{+}$ such that for any $ y \in \mathcal{Y}$, we have
\begin{align}
 \label{bnm}\widehat{\textbf G^w}^{\top}\odot (y-u) \ominus_{gH} c \lVert y-u \rVert \preceq{{\bf{\Phi}}}(y) \ominus_{gH} {\bf{\Phi}}(u).
 \end{align}
   From Lemma \ref{yuv}, we have
  $ - \lVert \widehat{\textbf G^w} \rVert_{I(\mathbb{R})^n} \lVert y-u \rVert - c \lVert y-u \rVert\preceq \widehat{\textbf G^w}\odot (y-u) \ominus_{gH} c \lVert y-u \rVert$. Hence, the inequality  (\ref{bnm}) yields
  \begin{align*}
     -(\lVert \widehat{\textbf G^w}\rVert+ c)\lVert y-u\rVert \preceq{{\bf{\Phi}}}(y) \ominus_{gH} {\bf{\Phi}}(u) ~\text{by Lemma 2.3 (ii) of \cite{anshika}}.
  \end{align*}
 By choosing $ L = (\lVert \widehat{\textbf G^w}\rVert + c )$, we obtain
\begin{align}
  -  L \lVert y- u \rVert\preceq{\bf{\Phi}}(y) \ominus_{gH} {\bf{\Phi}}(u)~ \forall~ y \in \mathcal Y.
\end{align}
So, ${\bf{\Phi}}$ is $gH$-lower Lipschitz at $u$.\\ \\
\fbox{(b) implies  (c)}~: Suppose that (b) is satisfied. It needs to prove that the inequality (\ref{ewr}) holds. Then, there exists an $L \geq 0 $ such that
\begin{align}\label{tr}
     &  -L \lVert y-u \rVert\preceq {\bf {\Phi}}(y) \ominus_{gH} {\bf{\Phi}}(u).
\end{align}
Note that $-L \lVert y \rVert -L \lVert u \rVert \leq -L \lVert y-u\rVert $. So,
the inequality (\ref{tr}) gives
\begin{align*}
-L \lVert y \rVert -L \lVert u \rVert \preceq{{\bf{\Phi}}}(y) \ominus_{gH} {\bf{\Phi}}(u),   \end{align*}
which gives
${\bf{\Phi}}(u) \ominus_{gH} L \lVert u \rVert - L \lVert y \rVert \preceq {{\bf{\Phi}}}(y) \text{ by (\ref{fds}) of Lemma \ref{yure}}.$
 Taking $\textbf Q ={\bf{\Phi}}(u)\ominus_{gH}  L \lVert u \rVert $ and $ p =  L$, we obtain $ -p\lVert y \rVert \oplus \textbf Q \preceq {{\bf{\Phi}}}(y)$ for all $y \in \mathcal Y$.\\ \\
\fbox{(c) implies (a)}~:
Let $\mathcal N (u)$ be an $\epsilon$-neighbourhood of $u$ such that (\ref{sdfhh}) holds. Then, we get
\begin{align}\label{cxz}
&   - L \lVert y -u \rVert\preceq{\bf{\Phi}}(y) \ominus_{gH}{\bf{\Phi}} (u)~\forall~ y \in  \mathcal N (u)
\end{align}
~$\text{and}$
\begin{align}\label{hgf}
&  -p \lVert y \rVert \oplus  \textbf{Q}\preceq{{\bf{\Phi}}}(y) ~\forall~ y \in \mathbb{R}^n.
\end{align}
Assume to the contrary that ${\bf{\Phi}}$ is not $gH$-weak subdifferentiable at $u$. Then, for any $(\widehat{\textbf G^{w}_{n}}, c
_{n}) \in I(\mathbb{R})^n \times \mathbb{R}_{+}$, there exists $y_{n}$ such that
\begin{align} \label{fght}
{\bf{\Phi}}(y_{n}) \ominus_{gH} {\bf {\Phi}}(u) \prec \widehat{\textbf G^{w}_{n}}^{\top} \odot (y_{n} - u) \ominus_{gH} c_{n} \lVert y_{n} - y \rVert. \nonumber\end{align}
If the sequence $\{\widehat{\textbf G^{w}_{n}}\} $ is assumed to be converging to $\widehat{\textbf G^w}$, then we get
\begin{align}
{\bf{\Phi}}(y_{n}) \ominus_{gH}  {{\bf{\Phi}}}(u) \preceq~& \widehat{\textbf G^w}^{\top} \odot (y_{n} - u) \ominus_{gH} c_{n} \lVert y_{n} - y \rVert\nonumber\\
\preceq ~&\lVert \widehat{\textbf G^w} \rVert \lVert y_{n}-u \rVert -c_{n} \lVert  y_{n}-u \rVert,~\text{by Theorem 3.1 of \cite{Ghosh2020lasso}}.
\end{align}
By putting $y=y_n$ in (\ref{hgf}), we get
\begin{align*}
 -p\lVert y_{n}-u\rVert -p \lVert y\rVert \oplus \textbf Q \preceq -p\lVert y_{n}\rVert \oplus \textbf Q\preceq{\bf{\Phi}}(y_{n}),
\end{align*}
which implies
\begin{align} \label{yhut}
   (-p \lVert y_{n}- u  \rVert -p \lVert u \rVert \oplus \textbf{Q})\ominus_{gH} {\bf{\Phi}}(u)\preceq{\bf{\Phi}}(y_{n}) \ominus_{gH} {\bf{\Phi}}(u) ~\text{by Note 2 of ~\cite{anshika}}.
\end{align}
From (\ref{fght}) and (\ref{yhut}), by Lemma 2.3 ~(ii) of \cite{anshika}, we deduce that
\begin{align}
& (-p \lVert y_{n}- u  \rVert -p \lVert u \rVert \oplus \textbf{Q})\ominus_{gH} {\bf{\Phi}}(u) \preceq \lVert \widehat{\textbf G^w} \rVert \lVert y_{n}-u \rVert -c_{n} \lVert  y_{n}-u \rVert,\nonumber\\
\text{or}, ~&(c_n-p-\lVert \widehat{\textbf G^w}\rVert)   \lVert y_{n} - u \rVert\preceq{\bf{\Phi}}(u)\oplus p\lVert u\rVert \ominus_{gH} \textbf{Q}  ~ \text{by ~(\ref{dbv}) ~of Lemma ~\ref{yure}}.
\end{align}
Assume, without loss of generality,  that $c_{n}-p-\lVert\widehat{\textbf G^w}\rVert \neq 0$. Then, from (\ref{yure}), we obtain
%\textcolor{green}
\[ \lVert y_{n} -u \rVert \preceq\frac{1}{c_{n}-p-\lVert\widehat{\textbf G^w}\rVert}\odot \{{\bf{\Phi}}(u)\oplus  p \lVert u \rVert  \ominus_{gH} \textbf Q\}.
\]
As $({\bf{\Phi}}(u) \oplus p \lVert u\rVert \ominus_{gH} \textbf Q)$ is bounded below on $\mathcal N(u)$, we get $y_{n} \to u$ as $c_{n} \to \infty$. Thus, $y_{n} \in \mathcal N(u)$ for large $n$. Then, from
(\ref{cxz}) it follows that
\begin{align}
   - L \lVert y_{n}-u \rVert \preceq{\bf{\Phi}}(y_{n}) \ominus_{gH} {\bf{\Phi}}(u).
\end{align}
In view of (\ref{fght}), we obtain
\begin{align*}
    &{\bf{\Phi}}(y_{n}) \ominus_{gH} {\bf{\Phi}}(u) \preceq \lVert \widehat{\textbf G^w}\rVert \lVert y_{n} - u\rVert -c_{n} \lVert y_{n}- u  \rVert
      = -(c_{n}-\lVert \widehat{\textbf G^w}\rVert ) \lVert y_{n}-u \rVert.
\end{align*}
Since $c_{n} \to +\infty $ and $ L \geq 0$, we can pick $c_{n}$  sufficiently large  %such
so that $c_{n}-\lVert \widehat{\textbf{G}^w}\rVert \geq L$.
So,\[{\bf{\Phi}}(y_n) \ominus_{gH} {\bf{\Phi}}(u) \preceq -L \lVert y_{n}-u\rVert.\]
This inequality leads to a contradiction. So, the result follows.
\end{proof}

\begin{theorem}
Let $\emptyset \neq \mathcal{Y} \subseteq \mathbb{R}^n$. Let ${\bf{\Psi}}: \mathcal{Y}\rightarrow I(\mathbb{R})$ be $gH$-Fr\'{e}chet differentiable at $u$ with $gH$-Fr\'{e}chet derivative ${\bf{\Psi}}_{\mathscr F}(u)$. Then,  \[ \{({\bf{\Psi}}_{\mathscr{F}}(u),c): c \geq 0\} \subset \partial^{w}{\bf{\Psi}}(u).\]
\end{theorem}
\begin{proof}
Since ${\bf{\Psi}}$ is $gH$-Fr\'{e}chet differentiable at $u$ with $gH$-Fr\'{e}chet derivative ${\bf{\Psi}}_{\mathscr F}(u)$, we get
\begin{align*}
 & \lim_{\substack{%
 y \rightarrow u}} \frac{1}{\lVert y-u \rVert}\odot \{{{\bf{\Psi}}}(y) \ominus_{gH} {\bf{\Psi}}(u) \ominus_{gH}{\bf{\Psi}}_{\mathscr F}(u)^{\top} \odot (y- u)\}=\textbf 0\\
   \implies & \liminf_{\substack{%
 y \rightarrow u\\ y \neq u}}\frac{1}{\lVert y-u\rVert}\odot\{{{\bf{\Psi}}}(y)\ominus_{gH} {\bf{\Psi}}(u) \ominus_{gH}{\bf{\Psi}}_{\mathscr F}(u)^{\top} \odot (y-u)\}=\textbf 0. ~%\text{by properties of gH-difference} ~(iv) ~ \cite{Tao}
 \end{align*}
 Therefore, by Definition \ref{werrr_1}, ${\bf{\Psi}}_{\mathscr{F}}(u) \in \partial^{-}_{\mathscr F} {\bf{\Psi}}(u)$. So,
 \begin{align*}
 & {\bf{\Psi}}_{\mathscr{F}}(u)^{\top}  \odot(y-u) \preceq {{\bf{\Psi}}}(y) \ominus_{gH} {\bf{\Psi}}(u)~ \forall~ y \in \mathcal{Y} \\
 \implies & {\bf{\Psi}}_{\mathscr{F}}(u)^{\top}  \odot(y-u) \ominus_{gH} c \lVert y-u \rVert\preceq{{\bf{\Psi}}}(y) \ominus_{gH} {\bf{\Psi}}(u), ~\text{for any} ~c\geq 0.
 \end{align*}
  Hence, $({\bf{\Psi}}_{\mathscr F}(u), c) \in \partial^{w} {\bf{\Psi}}(u)$.
 \end{proof}

 \begin{lemma}\label{fsg}
Let $\emptyset \neq \mathcal{Y} \subseteq \mathbb{R}^n$. Let ${\bf{\Phi}}: \mathcal{Y}\rightarrow I(\mathbb{R})$ be $gH$-Fr\'{e}chet differentiable at $u$ with $gH$-Fr\'{e}chet derivative ${\bf{\Phi}}_{\mathscr F}(u)$. Then,  $ -1 \odot{\bf{\Phi}}_{\mathscr{F}}(u) \in \partial^{-}_{\mathscr{F}}(-1 \odot{\bf{\Phi}})(u).$
\end{lemma}
\begin{proof}
Since ${\bf{\Phi}}$ is $gH$-Fr\'{e}chet differentiable at $u$ with $gH$-Fr\'{e}chet derivative ${\bf{\Phi}}_{\mathscr F}(u)$, one gets
\begin{align*}
 & \lim_{\substack{%
 y \rightarrow u}} \frac{1}{\lVert y-u \rVert}\odot\{{{\bf{\Phi}}}(y) \ominus_{gH} {\bf{\Phi}}(u) \ominus_{gH}{\bf{\Phi}}_{\mathscr F}(u)^{\top} \odot (y- u)\}=\textbf 0.
 \end{align*}
  By applying Lemma \ref{ldr1}, we have
  \footnotesize{
 \begin{align*}
 & \lim_{\substack{%
 y \rightarrow u\\ y \neq u}}\frac{1}{\lVert y-u\rVert}\odot \bigg\{\textbf 0\ominus_{gH} \{(-1\odot {{\bf{\Phi}}})(y) \ominus_{gH} (-1 \odot {{\bf{\Phi}})(u)} \ominus_{gH} (-1\odot{\bf{\Phi}}_{\mathscr F}(u)^{\top}) \odot (y-u)\}\}\bigg\}=\textbf 0\\
  ~%\text{by properties of gH-difference} ~(iv) ~ \cite{Tao}
 \implies & \lim_{\substack{%
 y \rightarrow u\\ y \neq u}}
 \frac{1}{\lVert y-u \rVert}\odot\bigg\{ (-1\odot {\bf{\Phi}})(y) \ominus_{gH} (-1 \odot {\bf{\Phi})(u)} \ominus_{gH} (-1\odot{\bf{\Phi}}_{\mathscr F}(u))^{\top} \odot (y-u)\bigg\}=\textbf 0 \\
 \implies & \liminf_{\substack{%
 y \rightarrow u\\ y \neq u}}
 \frac{1}{\lVert y-u \rVert}\odot\{ (-1\odot {\bf{\Phi}})(y) \ominus_{gH} (-1 \odot {\bf {\Phi})(u)} \ominus_{gH} (-1\odot{\bf{\Phi}}_{\mathscr F}(u))^{\top} \odot (y-u)\}=\textbf 0.
 \end{align*}
 }
Hence, $-1 \odot{\bf{\Phi}}_{\mathscr F}(u) \in \partial^{-}_{\mathscr F}(-1 \odot {\bf{\Phi}})(u)$.
\end{proof}

 Next, we focus on investigating the sum rule of two functions in terms of $gH$-weak subdifferential. For two real-valued functions $f_1$ and $f_2$, the sum rule \cite{Inceoglu} for their weak subdifferential is $\partial^{w}(f_1 + f_2)(x)= \partial^{w} f_1(x)+ \partial^{w} f_2(x).$ However, this sum rule
 does not hold for interval-valued functions. In the following, we provide such an example.

Consider the interval-valued functions ${\bf{\Phi}}_{1}:[-1,1]\to I(\mathbb R) $ and ${\bf{\Phi}}_{2}: [-1,1] \to I(\mathbb R)$, defined by
\[
{\bf{\Phi}}_{1}(y)=
 \begin{cases}
 \left[-y, \frac{1}{2}y\right] & ~\text{if}~ y \in [0,1]\\
  \left[-\frac{1}{2}y -y\right] & ~\text{if}~ ~y \in [-1,0]
 \end{cases}
~\text{ and }~{\bf{\Phi}}_{2}(y)=[y^2,-y+3],\] respectively.
For these two functions, the $gH$-weak subdifferential at $u =0$ are given by
\[\partial^{w}{\bf{\Phi}}_{1}(0) = \{ ({\textbf G^{w}_{1}},c_1) \in I(\mathbb{R}) \times \mathbb{R}_{+}: \left[-1,-\tfrac{1}{2}\right] \preceq {\textbf G^{w}_{1}} \oplus  c_{1},~ {\textbf G^{w}_1}\ominus_{gH} c_{1} \preceq\left[-1, \tfrac{1}{2}\right]~\forall~ y \in [-1,1]\} \]
and
\[\partial^{w}{\bf{\Phi}}_{2}(0) = \{({\textbf G^{w}_{2}},c_{2}) \in I(\mathbb R) \times \mathbb R_{+}: [-1,0] \preceq {\textbf G^{w}_{2}}\oplus c_{2},~ {\textbf G^{w}_2}\ominus_{gH}c_{2} \preceq [-1,0]~\forall ~y \in [-1,1]\}. \]
Thus, we have
\begin{align} \label{weaksum_1}
&\partial^{w}{\bf{\Phi}}_{1}(0)\oplus\partial^{w}{\bf{\Phi}}_{2}(0) \nonumber\\
=~& \{({\textbf H^w}, c) \in I(\mathbb R) \times \mathbb R_{+}: \left[-2,-\tfrac{1}{2}\right] \preceq {\textbf H^w}\oplus c,~{\textbf H^w}\ominus_{gH} c \preceq \left[-2,\tfrac{1}{2}\right]~\forall~ y \in [-1,1]\}.
\end{align}
Now, let $({\textbf{H}^w},c) \in \partial^{w} ({\bf{\Phi}}_{1}\oplus{\bf{\Phi}}_{2}\}(0)$, where
\[({\bf{\Phi}}_{1} \oplus{\bf{\Phi}}_{2})(y)=
\begin{cases}
 \left[y^2-y, -\frac{1}{2}y+3\right] &\text{if}   ~y\in [0,1]\\
 \left[y^2-\frac{1}{2}y, -2y+3\right] &\text{if}~ y \in [-1,0].
\end{cases} \]
 There are the following two cases corresponding to  $y \in [0,1] $ and $y \in [-1,0]$.
 \begin{enumerate}[$\bullet$ \textbf{Case} 1.]
 \item \label{bttb}As $y \geq 0$, we have
 \begin{eqnarray*}
 &&   {\textbf H^w}\odot y\ominus_{gH}c \odot y\preceq ({\bf{\Phi}}_{1}\oplus{\bf{\Phi}}_{2})(y) \ominus_{gH} ({\bf{\Phi}}_{1}\oplus {\bf{\Phi}}_{2})(0)\nonumber\\
 &\implies&  [\underline {h^{w}}-c, \overline {h^w}-c]\odot y\preceq \left[y^2-y, -\tfrac{1}{2}y\right] \nonumber\\
 & \implies &  \underline {h^w} -c\leq -1 ~\text{and }~  \overline {h^w} -c \leq-\tfrac{1}{2}.
  \end{eqnarray*}
 \item \label{vcx}As $-1 \leq y \leq 0 $, we have
 \begin{eqnarray*}
  &&  [(\overline {h^w} +c)y, (\underline {h^w}+c)y]\preceq\left[y^2-\tfrac{1}{2}y,-2y+3\right] \ominus_{gH} [0,3] \nonumber\\
  &\implies & [(\overline {h^w} +c)y, (\underline {h^w} +c)y]\preceq\left[ y^2-\tfrac{1}{2}y, -2y\right]\nonumber\\
  &\implies& -2-c \leq \underline {h^w} ~\text{and} -\tfrac{1}{2}-c \leq  \overline {h^w}.
 \end{eqnarray*}
 \end{enumerate}
 Therefore, form \textbf{Case} \ref{bttb} and \textbf{Case} \ref{vcx}, we have
 \begin{align}\label{weaksum_2}
 &\partial^{w}({\bf{\Phi}}_{1}\oplus{\bf{\Phi}}_{2})(0) \nonumber\\
 = ~&\{({\textbf H^w},c) \in I(\mathbb R) \times \mathbb R_{+}: \left[-2, -\tfrac{1}{2}\right] \preceq ({\textbf H^w}\oplus c), ({\textbf H^w}\ominus_{gH} c) \preceq \left[-1, -\tfrac{1}{2} \right] \forall~ y \in [0,1]\}.\end{align}
 Thus, (\ref{weaksum_1}) and (\ref{weaksum_2}) are not equal. \\

In the following  Theorem \ref{onesidesumrule}, we show that under some restriction on ${\bf{\Phi}}_1$ and ${\bf{\Phi}}_2$ one-sided inclusion for the sum rule holds.

{\begin{theorem}\label{onesidesumrule}
Let $\emptyset \neq \mathcal{Y} \subseteq \mathbb{R}^n$.
Let ${\bf{\Phi}}_{1}: \mathcal{Y} \rightarrow I(\mathbb{R})$ be $gH$-weak subdifferential at $u$ and ${\bf{\Phi}}_{2}: \mathcal{Y} \rightarrow  \mathbb{R}$ be $gH$-Fr\'{e}chet differentiable at $u$. Then,
\[
    \partial^{w}({\bf{\Phi}}_{1}\oplus{\bf{\Phi}}_{2})(u) \subset \partial^{w}{\bf{\Phi}}_{1}(u) \oplus \partial^{w}{\bf{\Phi}}_{2}(u),
\]
provided that $w(\widehat{\textbf G^{w}_{1}})\leq w(\widehat{\textbf G^{w}_{2}})$ for all $\widehat{\textbf G^{w}_{1}} \in {\partial{\bf{\Phi}}_{2}(y)}$ and $\widehat{\textbf G^{w}_{2}} \in {\partial} ({\bf{\Phi}}_{1}\oplus{\bf{\Phi}}_{2})(y)$,
 where $w(\textbf A)$ is the width of the interval $\textbf A \in I(\mathbb R)$.
\end{theorem}
\begin{proof}
If $(\widehat{\textbf G^{w}},c) \in \partial^{w}({\bf{\Phi}}_{1}\oplus{\bf{\Phi}}_{2})(u)$, then
\begin{align}
 \widehat{\textbf G^{w}}^{\top} \odot (y-u) \ominus_{gH} c \lVert y-u\rVert  \preceq ({\bf{\Phi}}_{1} \oplus{\bf{\Phi}}_{2})(y) \ominus_{gH} ({\bf{\Phi}}_{1} \oplus{\bf{\Phi}}_{2})(u) \label{rnh}.
 \end{align}
We know that ${\bf{\Phi}}_{2}:\mathcal{Y} \rightarrow I(\mathbb{R})$ is $gH$-Fr\'{e}chet differentiable at $u$ with the $gH$-Fr\'echet derivative ${\bf{\Phi}}_{2\mathscr F}(u)$. Hence, ${\bf{\Phi}}_{2\mathscr F}(u) \in \partial^{-}_{\mathscr F}{\bf{\Phi}}_{2}(u)$ implies $-1 \odot{\bf{\Phi}}_{2\mathscr F}(u) \in \partial^{-}_{\mathscr F}(-1 \odot {\bf{\Phi}}_{2})(u)$. We can then write
\begin{align}
 &-1\odot{\bf{\Phi}}_{2\mathscr F}(u)\odot (y-u) \preceq (-1\odot {\bf{\Phi}}_{2})(u)\ominus_{gH} (-1\odot{\bf{\Phi}}_{2})(u)\nonumber\\
 \implies &-1 \odot{\bf{\Phi}}_{2\mathscr F}(u) \odot (y-u) \preceq -1 \odot({\bf{\Phi}}_{2}(y) \ominus_{gH}{\bf{\Phi}}_{2}(u))\nonumber\\
 &~\text{by properties of $gH$-difference~(iv) of \cite{Tao}} \label{nds}.
\end{align}
 In view of Lemma \ref{gh2}, (\ref{rnh}) becomes
\begin{align*}
  & \widehat{\textbf G^w}^{\top}\odot (y-u) \ominus_{gH} c\lVert y-u\rVert\preceq( {\bf{\Phi}}_{1}(y)\ominus_{gH}{\bf{\Phi}}_{1}(u)) \oplus({\bf{\Phi}}_{2}(y) \ominus_{gH}{\bf{\Phi}}_{2}(u))\nonumber.
 \end{align*}
 Using (\ref{ewn}) of Lemma \ref{yure}, this  inequality reduces to
 \begin{align*}
  & \widehat{\textbf G^{w}}^{\top}\odot (y-u) \ominus_{gH} ({\bf{\Phi}}_{2}(y)\ominus_{gH} {\bf{\Phi}}_{2}(u))\ominus_{gH} c\lVert y-u\rVert \preceq{\bf{\Phi}}_{1}(y) \ominus_{gH}{\bf{\Phi}}_{1}(u).
\end{align*}
Now, from the inequality (\ref{nds}), we see that
\begin{align*}
 \widehat{\textbf G^{w}}^{\top}\odot (y-u)\ominus_{gH}{\bf{\Phi}}_{2\mathscr F}(u)\odot (y-u)\ominus_{gH} c\lVert y-u\rVert \preceq{\bf{\Phi}}_{1}(y) \ominus_{gH}{\bf{\Phi}}_{1}(u).
\end{align*}
Thus,
\[(\widehat{\textbf G^w} \ominus_{gH} {\bf{\Phi}}_{2\mathscr F}(u))^{\top} \odot (y-u) \ominus_{gH} c \lVert y-u\rVert \preceq{\bf{\Phi}}_{1}(y)\ominus_{gH}{\bf{\Phi}}_{1}(u).
\]
Then, $(\widehat{\textbf G^w} \ominus_{gH}{\bf{\Phi}}_{2\mathscr F}(u),c) \in \partial^{w}{\bf{\Phi}}_{1}(u)$ and $({\bf{\Phi}}_{2 \mathscr F}(u),0) \in \partial^{w}{\bf{\Phi}}_{2}(u)$. Therefore, $(\widehat{\textbf G^w},c)\in \partial^{w}{\bf{\Phi}}_{1}(u)\oplus \partial^{w}{\bf{\Phi}}_{2}(u)$.
Hence, the result follows.
\end{proof}}

%%%%%%%%%%%%%%%%%%%%%%%%%%%%%%%%%%%%%%%%%%%%%%%%%%%%%%%%%%%%%%%%%%%%%%%%%%%%%%%%%%%%
\begin{theorem}
Let $\mathcal{Y}$ be a nonempty set of $\mathbb{R}^{n}$. Let  ${\bf{\Phi}}_{1}:\mathcal{Y} \rightarrow I(\mathbb{R})$ be $gH$-Fr\'{e}chet differentiable at $u
$. Let ${\bf{\Phi}}_{2}: \mathcal{Y} \rightarrow I(\mathbb{R})$ be an IVF. If $u$ is a weak efficient point of ${\bf{\Phi}}_{1}\oplus{\bf{\Phi}}_{2}$, then $(-1 \odot{\bf{\Phi}}_{1 \mathscr F}(u), 0) \in \partial^{w}{\bf{\Phi}}_{2}(u)$.
\end{theorem}
\begin{proof}
Since $u$ is a weak efficient point of ${\bf{\Phi}}_{1}\oplus{\bf{\Phi}}_{2} $, for any $ y \in \mathcal{Y}$,
\begin{align}
 &  ({\bf{\Phi}}_{1} \oplus{\bf{\Phi}}_{2}) (u)\preceq   ({\bf{\Phi}}_{1}\oplus{\bf{\Phi}}_{2})(y)\nonumber\\
 \implies & {\bf{\Phi}}_{1}(u) \oplus{\bf{\Phi}}_{2}(u)\preceq{\bf{\Phi}}_{1}(y) \oplus{\bf{\Phi}}_{2}(y) \nonumber\\
\implies &{\bf{\Phi}}_{1}(u) \ominus_{gH}{\bf{\Phi}}_{1}(y)\preceq{\bf{\Phi}}_{2}(y) \ominus_{gH}{\bf{\Phi}}_{2}(u),~\text{using Lemma 2.3 of \cite{anshika}} \nonumber\\
  \implies &  (-1)\odot \{{\bf{\Phi}}_{1}(y) \ominus_{gH}{\bf{\Phi}}_{1}(u)\} \preceq{\bf{\Phi}}_{2}(y) \ominus_{gH}{\bf{\Phi}}_{2}(u),~ \text{by} \ominus_{gH} \text{property in (iv) of \cite{Tao}}\nonumber\\
 \implies & (-1 \odot{\bf{\Phi}}_{1})(y) \ominus_{gH} (-1\ \odot{\bf{\Phi}}_{1})(u)\preceq{\bf{\Phi}}_{2}(y)\ominus_{gH}{\bf{\Phi}}_{2}(u), \nonumber\\
 &~\text{by} \ominus_{gH} \text{property in (iv) of \cite{Tao}}.
\label{bvd}
  \end{align}
 By the Lemma \ref{fsg}, we also obtain that
 \begin{align}\label{wqbgfre}
   (-1)\odot{\bf{\Phi}}_{1\mathscr F}(u) \odot (y-u)\preceq (-1 \odot{\bf{\Phi}}_{1})(y) \ominus_{gH} (-1\ \odot{\bf{\Phi}}_{1})(u)~ \forall~ y \in \mathcal{Y}.
\end{align}
 We get, from (\ref{bvd}) and (\ref{wqbgfre}) that
\begin{align*}
(-1) \odot{\bf{\Phi}}_{1\mathscr F}(u) \odot (y-u) \preceq{\bf{\Phi}}_{2}(y) \ominus_{gH}{\bf{\Phi}}_{2}(u)~\text{by Lemma 2.3 ~(ii)~of \cite{anshika}}.
    \end{align*}
This means that $((-1) \odot{\bf{\Phi}}_{1\mathscr F} (u),0) \in \partial^{w}{\bf{\Phi}}_{2}(u)$.
\end{proof}

\begin{theorem}
 Let $\emptyset \neq \mathcal{Y} \subseteq \mathbb{R}^n$. Let ${\bf{\Psi}}$ be $gH$-Fr\'{e}chet differentiable  at $u$ with the $gH$-Fr\'{e}chet derivative ${\bf{\Psi}}_{\mathscr F}(u)$. Then, ${\bf{\Psi}}$ has weak efficient solution at $u$ if and only if for any $y \in \mathcal{Y}$,
 \[ {\bf{\Psi}}_{\mathscr F}(u)^{\top} \odot (y-u)=\textbf 0.\]
\end{theorem}
\begin{proof}
If ${\bf{\Psi}}$ has a weak efficient point at $u$, then
\begin{align*}
       &   {\bf{\Psi}}(u)\preceq {{\bf{\Psi}}}(y)\\
    \text{or}, ~ & \textbf 0 \preceq {{\bf{\Psi}}}(y) \ominus_{gH} {\bf{\Psi}}(u),\text{ by Lemma 2.1}\text{ of \cite{Ghosh2019derivative}}.
     \end{align*}
By $gH$-Fr\'{e}chet differentiability of $\bf{\Psi}$ at $u$, we get
\[\lim_{\substack{%
\lVert h \rVert \to 0}}\frac{\lVert ({\bf{\Psi}}(u+h)\ominus_{gH} {{\bf{\Psi}}}(u)) \ominus_{gH} {\bf{\Psi}}_{\mathscr F}(u)^\top\odot h\rVert_{I(\mathbb{R})}}{\lVert h\rVert}=0.
\]
If we take $h= \lambda(y-u)$, then
\begin{align}
\lim_{\substack{%
 \lambda \to 0}}\frac{\lVert ({\bf{\Psi}}(u+\lambda(y-u))\ominus_{gH} {\bf{\Psi}}(u)) \ominus_{gH}{\bf{\Psi}}_{\mathscr F}(u)^\top \odot \{\lambda(y-u)\}\rVert_{I(\mathbb{R})}}{\lVert \lambda(y-u)\rVert}=0. \label{hh}
\end{align}
Since $u$ is a weak efficient point of ${\bf{\Psi}}$, from (\ref{hh}) we have
\begin{eqnarray*}
&&\lim_{\substack{%
\lambda \to 0}}\frac{\lVert \textbf 0\ominus_{gH}\lambda \odot \{{\bf{\Psi}}_{\mathscr F}(u)^\top \odot (y-u)\}\rVert_{I(\mathbb{R})}}
{\lVert \lambda (y-u)\rVert} \leq 0 ~ \text{by~(\ref{3_1})} ~\text{of Lemma~\ref{yure}}\\
&\implies&   \lim_{\substack{%
\lambda \to 0}}\frac{\lVert \lambda \odot \{{\bf{\Psi}}_{\mathscr F}(u)^\top \odot (y-u)\}\rVert_{I(\mathbb{R})}}
{\lVert \lambda (y-u)\rVert} \leq 0\nonumber\\
&\implies& \lim_{\substack{%
\lambda \to 0}}\frac{\lambda\lVert{\bf{\Psi}}_{\mathscr F}(u)^\top \odot (y-u)\rVert_{I(\mathbb{R})}}
{\lambda\lVert (y-u)\rVert} \leq 0\nonumber.
\end{eqnarray*}
Since norm gives non-negative value,
\begin{align*}
 \frac{1}{\lVert y-u\rVert}\odot\{{\bf{\Psi}}_{\mathscr F}(u)^\top\odot (y-u) \}=\textbf 0.
\end{align*}
Thus, we obtain
\begin{align*}
 &{\bf{\Psi}}_{\mathscr F}(u)^\top \odot(y-u) =\textbf 0~\text{for any}~y \in \mathcal Y.
\end{align*}
To show the reverse part, we suppose that ${\bf{\Psi}}_{\mathscr F}(u)^\top \odot(y-u) =\textbf 0$ for all $y$. Then, we have ${\bf{\Psi}}_{\mathscr{F}}(u) \in \partial^{-}_{\mathscr F}  {\bf{\Psi}}(u)$ and this clearly yields
\begin{align*}
 &\textbf 0 ={\bf{\Psi}}_{\mathscr F}(u)^{\top} \odot (y-u)\preceq {{\bf{\Psi}}}(y) \ominus_{gH} {\bf{\Psi}}(u) \\
\implies  & {\bf{\Psi}}(u) \preceq {{\bf{\Psi}}}(y) ~ \text{by (ii) of Lemma 2.1} ~ \text{in \cite{Ghosh2019derivative}},
\end{align*}
and this means that $u$ is weak efficient point of $\textbf T$.
\end{proof}

\begin{theorem}\label{opu}
Let $\emptyset \neq \mathcal{Y} \subseteq \mathbb{R}^n$. If ${\bf{\Psi}}$ is $gH$-Fr\'{e}chet differentiable at $u$, then ${\bf{\Psi}}$ is $gH$-weak subdifferentiable at $u$ if and only if ${\bf{\Psi}}_{\mathscr F}(u)$ is  $gH$-weak subdifferentiable at $0 \in \mathcal Y$, and
\[  \partial^{w}({\bf{\Psi}}(u))= \partial^{w}({\bf{\Psi}}_{\mathscr F}(u )(0)).
\]
\end{theorem}
\begin{proof}
By the $gH$-Fr\'{e}chet differentiability of ${\bf{\Psi}}$ at $u$, we have
\[\lim_{\substack{%
\lVert h \rVert \to 0}}\frac{1}{\lVert h\rVert} \odot\{ ({\bf{\Psi}}(u+h)\ominus_{gH} {\bf{\Psi}}(u)) \ominus_{gH}{\bf{\Psi}}_{\mathscr F}(u)^\top \odot h\}=\textbf 0.
\]
 Inserting $h = \lambda \odot (y-u)$, by $gH$-weak subdifferentiability of ${\bf{\Psi}}$ at $u$, there exists ($\widehat{\textbf{G}^w},c)\in \partial^{w}{\bf{\Psi}}(u)$ such that for any $y \in  \mathcal{Y}$,
\[ \widehat{\textbf G^w}^{\top} \odot (y-u) \ominus_{gH} c \lVert y-u \rVert \preceq {{\bf{\Psi}}}(y) \ominus_{gH} \textbf{T}(u).
\]
Hence, \[\lim_{\substack{%
\lambda \to 0}}\frac{1}{\lVert \lambda (y-u)\rVert}\odot\{({\bf {\Psi}}(u+ \lambda (y-u))\ominus_{gH} {\bf{\Psi}}(u))\ominus_{gH}{\bf{\Psi}}_{\mathscr F}(u)^\top\odot \lambda(y-u)\}=\textbf 0\]
and by $gH$-weak subdifferentiability of ${\bf{\Psi}}$ at $u$, we get, for any $y \in \mathcal{Y}$ that
\begin{align*}
&\lim_{\substack{%
\lambda \to 0}} \frac{1}{\lVert \lambda (y-u)\rVert} \odot\bigg\{(\widehat{\textbf G^w}^{\top}\odot \lambda (y-u) \ominus_{gH} \lambda c \lVert y-u \rVert )\ominus_{gH}{\bf{\Psi}}_{\mathscr F}(u)^\top \odot \lambda(y-u)\bigg\} \preceq \textbf 0,~ \\  & (\text{by~(\ref{cv}) of Lemma \ref{yure}})\\
\implies &  \frac{1}{\lVert  (y-u)\rVert}\odot\{(\widehat{\textbf G^w}^{\top}\odot  (y-u) \ominus_{gH}  c \lVert y-u \rVert )\ominus_{gH}{\bf{\Psi}}_{\mathscr F}(u)^\top \odot (y-u)\} \preceq \textbf 0.
\end{align*}
Therefore,  \[ \widehat{\textbf G^w}^{\top} \odot (y-u) \ominus_{gH} c \lVert y-u \rVert \ominus_{gH} {\bf{\Psi}}_{\mathscr F}(u)^\top \odot (y-u) \preceq \textbf 0
~~ \forall~ y \in \mathcal{Y}\]
and so by letting  $z= y-u$, we obtain
  \begin{align}
  \label{wrr} \widehat{\textbf G^w}^{\top} \odot z \ominus_{gH}c \lVert z \rVert
  \preceq{\bf{\Psi}}_{\mathscr F}(u)^\top \odot z~ \forall~ z \in \mathcal{Y}.
  \end{align}
  Note that the $gH$-Fr\'{e}chet derivative ${\bf{\Psi}}_{\mathscr F}(u)$ is also $gH$-G\'ateaux derivative as in (see Theorem 5.2 of \cite{Ghosh2019derivative}). Hence, it is a linear IVF as in Definition 4.1 of \cite{Ghosh2019derivative}. By this fact, we have ${\bf{\Psi}}_{\mathscr F}(u)^\top \odot (0)=\textbf 0$. Then, the inequality (\ref{wrr}) implies that $(\widehat{\textbf G^w},c) \in \partial^{w}({\bf{\Psi}}_{\mathscr F}(u)(0))$.\\
 Conversely, let $(\widehat{\textbf G^w}, c) \in \partial^{w}({\bf{\Psi}}_{\mathscr F}(u)(0))$. Then, we can write
  \begin{align}
 &  \widehat{\textbf G^w}^{\top} \odot y \ominus_{gH}c \lVert y \rVert \preceq{\bf{\Psi}}_{\mathscr F}(u)^{\top}\odot y ~\forall~ y \in \mathcal{Y}\nonumber\\
  \implies&  \widehat{\textbf G^w}^{\top} \odot (y-u) \ominus_{gH}c \lVert y-u \rVert\preceq{\bf{\Psi}}_{\mathscr F}(u)^\top\odot(y-u)\nonumber~ \forall~ y \in \mathcal{Y}.
  \end{align}
  Since ${\bf{\Psi}}$ has $gH$-Fr\'{e}chet derivative ${\bf{\Psi}}_{\mathscr F}(u)$ and it is also a $gH$-subgradient, it follows that
  \[{\bf{\Psi}}_{\mathscr F}(y)^{\top} \odot (y-u) \preceq {{\bf{\Psi}}}(y) \ominus_{gH} {\bf{\Psi}}(u)~ \forall~ y \in \mathcal{Y}.
  \]
  Then, $ \widehat{\textbf G^w}^{\top}\odot (y-u) \ominus_{gH} c \lVert y-u\rVert\preceq{{\bf{\Psi}}}(y) \ominus
  _{gH} {\bf{\Psi}}(u)  $. Hence the proof is complete.
 \end{proof}

\begin{theorem} \label{dsa}
Let $\emptyset \neq \mathcal{Y} \subseteq \mathbb{R}^n$. Let ${\bf{\Phi}}$ is $gH$-Fr\'{e}chet differentiable at $u$. If $u$ is a weak efficient point of ${\bf{\Phi}}$, then
\[ \text{sup}\left\{ \widehat{\textbf G^w}^{\top}\odot (y-u) \ominus_{gH} c \lVert y-u\rVert : (\widehat{\textbf G^w},c)\in \partial^{w}{\bf{\Phi}}(u)\right\}=\textbf {0}.
\]
\end{theorem}
\begin{proof}
First, we show that
 \[{\bf{\Phi}}_{\mathscr F}(u)^{\top}\odot (y-u)= \text{sup}\left\{ \widehat{\textbf G^w}^{\top}\odot (y-u) \ominus_{gH} c \lVert y-u\rVert : (\widehat{\textbf G^w},c) \in \partial^{w} {\bf{\Phi}}(u)\right\}
 \] by which the desired equality can be easily proved.
 By $gH$-Fr\'echet differentiability of $\textbf{T}$ and by taking  the supremum on the inequality (\ref{wrr}), we obtain
\begin{align*}
  \sup_{\substack{%
  (\widehat{\textbf G^w}, c)\in \partial^{w}{\bf{\Phi}}(u)}}\{ \widehat{\textbf G^w}^{\top}\odot (y-u) \ominus_{gH} c \lVert y-u\rVert \}\preceq&\sup_{\substack{%
  (\widehat{\textbf G^w}, c)\in \partial^{w}\textbf{T}(u)}}\{{\bf{\Phi}}_{\mathscr F}(u)^{\top}\odot (y-u)\}  \nonumber
   \\
  ~~=&~~{\bf{\Phi}}_{\mathscr F}(u)^{\top}\odot (y-u).
  \end{align*}
 Since $({\bf{\Phi}}_{\mathscr F}(u),0) \in \partial^{w} \textbf{T}(y)$,
 \[{\bf{\Phi}}_{\mathscr F}(u)\odot (y-u) \in \left\{ \widehat{\textbf G^w}^{\top}\odot(y-u)\ominus_{gH} c \lVert y-u\rVert: (\widehat{\textbf G^w},c) \in \partial^{w}{\bf{\Phi}}(u) \right\}\]
  and hence the result follows.
\end{proof}

%%%%%%%%%%%%%%%%%%%%%%%%%%%%%%%%%%%%%%%%%%%%%%%%%%%%%%%%%%%%%%%%%%%%%%%%%%%%%%%%%%%%
%%%%%%%%%%%%%%%%%%%%%%%%%%%%%%%%%%%%
\section{Optimality for the difference of two IVFs}\label{section5}
In this section,
we consider the constrained IOP as below:
\begin{align}\label{yfhmmvc}
   \min_{ y \in {\mathcal Y}} ~\{{\bf{\Phi}}_{2}( y) \ominus_{gH}{\bf{\Phi}}_{1}(y) \},
 \end{align}
 where $\emptyset \neq \mathcal{Y} \subseteq \mathbb{R}^n$  and  ${\bf{\Phi}}_{1},{\bf{\Phi}}_{2}: {\mathcal Y} \to I(\mathbb{R})$ are two  IVFs. We are going to  study of weak efficiency condition for the IOP (\ref{yfhmmvc})  under some additional assumptions.

\begin{theorem}\label{tyr}
Let $\emptyset \neq \mathcal Y \subseteq \mathbb{R}^n$. Let ${\bf{\Phi}}_{1},{\bf{\Phi}}_{2}: \mathcal Y \rightarrow I(\mathbb{R}) $ be $gH$-weak subdifferentiable at $u$, which is a weak-efficient point of ${\bf{\Phi}}_{2} \ominus_{gH}{\bf{\Phi}}_{1}$. If   ${\bf{\Phi}}_{1}(u)={\bf{\Phi}}_{2}(u)$, then
\[ \partial^{w}{\bf{\Phi}}_{1}(u) \subset \partial^{w}{\bf{\Phi}}_{2} (u). \]
\end{theorem}
\begin{proof}
The $gH$-weak subdifferentiability of ${\bf{\Phi}}_{1}$ at $u$ implies that $\partial^{w}{\bf{\Phi}}_{1}(u) $ is nonempty. Hence, there exists $(\widehat{\textbf U^w},c ) \in I(\mathbb{R}) \times \mathbb{R}_{+}$ such that
\begin{align}
   \widehat{\textbf U^w}^{\top} \odot (y-u) \ominus_{gH} c \lVert y-u \rVert \preceq{\bf{\Phi}}_{1}(y) \ominus_{gH}{\bf{\Phi}}_{1}(u) ~\text{for all}~ y \in \mathcal Y.
\end{align}
Since ${\bf{\Phi}}_{2} \ominus_{gH} {\bf{\Phi}}_{1}$ gets the weak efficiency value $\textbf 0$ at $u$, for any $y \in \mathcal Y$, we have
\begin{align}
   &  \textbf 0 \preceq({\bf {\Phi}}_{2}\ominus_{gH}{\bf{\Phi}}_{1})(y)  \nonumber\\
   \implies & \textbf 0 \preceq {\bf {\Phi}}_{2}(y)\ominus_{gH} {\bf {\Phi}}_{1}(y)\nonumber\\
 \implies  &  {\bf {\Phi}}_{1}(y)\preceq{\bf{\Phi}}_{2}(y) \nonumber ~\text{by Lemma 2.1(ii) of \cite{Ghosh2019derivative}} \\
  \implies & {\bf{\Phi}}_{1}(y) \ominus_{gH}{\bf{\Phi}}_{1}(u) \preceq{\bf{\Phi}}_{2}(y) \ominus_{gH}{\bf{\Phi}}_{2}(u)  \label{nxc}~\text{by Note 2 of \cite{anshika}}.
\end{align}
Consequently, the inequality~(\ref{nxc}) implies that
\[   \widehat{\textbf U^w}^{\top} \odot (y-u) \ominus_{gH} c\lVert y-u \rVert \preceq{\bf{\Phi}}_{2}(y) \ominus_{gH}{\bf{\Phi}}_{2}(u).
\]
This means $(\widehat{\textbf U^w},c )
\in \partial^{w}{\bf{\Phi}}_{2}(y)$. Hence, the result follows.\\
\end{proof}

\begin{note} If we had taken an efficient solution of ${\bf{\Phi}}_{2} \ominus_{gH}{\bf{\Phi}}_{1}$ instead of a weak efficient solution,
 the additional condition ${\bf {\Phi}}_{1}(u)={\bf{\Phi}}_{2}(u)$  becomes essential for  Theorem \ref{tyr} to hold. For instance, let two IVFs $\bf{\Phi}_{1}: \left[-\frac{1}{2}, \frac{1}{2}\right] \rightarrow I(\mathbb{R})$ and ${\bf{\Phi}}_{2}: \left[-\frac{1}{2},\frac{1}{2}\right] \rightarrow  I(\mathbb{R})$  be defined as
\[{\bf{\Phi}}_{1}(y) =[2 \lvert y\rvert, \lvert y\rvert +1] \text{ and } {\bf{\Phi}}_{2}(y)=[\lvert y \rvert, 2y^2+ \lvert y\rvert],\] respectively. Now, according to Theorem \ref{tyr}, $({\bf{\Phi}}_{2}\ominus_{gH}{\bf{\Phi}}_{1})(y)=[2y^2-1, -\lvert y\rvert]$, and $0$ is an efficient point of $({\bf{\Phi}}_{2} \ominus_{gH}{\bf{\Phi}}_{1})$ because $({\bf{\Phi}}_{2}\ominus_{gH}{\bf{\Phi}}_{1})(y)$ and $({\bf{\Phi}}_{2} \ominus_{gH}{\bf{\Phi}}_{1})(0)$ are not comparable for all $y \in \left[-\frac{1}{2}, \frac{1}{2}\right]$.
Note that
\begin{align*}
&\partial^{w}{\bf{\Phi}}_{1}(0)=\{({\textbf K^{w}_{1}},c_{1}): [-2,-1] \preceq ({\textbf K^{w}_{1}}\oplus c_{1}),({\textbf K^{w}_{1}}\ominus_{gH}c_{1}) \preceq [1,2] \}\\
~\text{and}~&\partial^{w}{\bf{\Phi}}_{2}(0)=\{ ({\textbf K^{w}_{2}},c_{2}):[-1,-1]\preceq ({\textbf K^{w}_{2}}\oplus c_{2}), ({\textbf K^{w}_{2}}\ominus_{gH}c_{2}) \preceq [1,1]\}
\end{align*}
Hence, $\partial^{w}{\bf{\Phi}}_{1}(0) \not\subset \partial^{w}{\bf{\Phi}}_{2}(0) $. So, ${\bf{\Phi}}_{1}(u)={\bf{\Phi}}_{2}(u)$ is an essential condition.\\
\end{note}

As the restriction ${\bf{\Phi}}_{1}(u)={\bf{\Phi}}_{2}(u)$ is a bit restrictive, in the next result, we give more flexible condition for which the inclusion in Theorem \ref{tyr} holds.

\begin{theorem}\label{plk}
Let $\emptyset \neq \mathcal{Y} \subseteq \mathbb{R}^n$. Let ${\bf{\Phi}}_{1},{\bf{\Phi}}_{2}$  have $gH$-weak subdifferential at $u \in \mathcal Y $, and ${\bf{\Phi}}_{2}\ominus_{gH}{\bf{\Phi}}_{1}$ attains  weak efficient solution at  $u$. Then,
\begin{align}\label{dini2}
 \partial^{w}{\bf{\Phi}}_{1}(u) \subset \partial^{w}{\bf{\Phi}}_{2}(u)
\end{align}
provided that $w({\bf{\Phi}}_{1}(y)) \geq w({\bf{\Phi}}_{2}(y))$ for $y \in \mathcal{Y}$ or $w({\bf{\Phi}}_{1}(y)) \leq w({\bf{\Phi}}_{2}(y))$ for $y \in \mathcal{Y}$, where $w(\textbf A)$ is the width of the interval $\textbf A \in I(\mathbb R)$.
\end{theorem}
\begin{proof}
The $gH$-weak subdifferentiability of ${\bf{\Phi}}_{1}$ at $u$ implies that $\partial^{w} {\bf{\Phi}}_{1}(u) $ is nonempty. Hence, there exists $(\widehat{\textbf U^w},c ) \in I(\mathbb{R}) \times \mathbb{R}_{+}$ such that
\begin{align}\label{er}
    \widehat{\textbf U^w}^{\top} \odot (y-u) \ominus_{gH} c \lVert y-u \rVert\preceq{\bf{\Phi}}_{1}(y) \ominus_{gH}{\bf{\Phi}}_{1}(u) ~\text{for all}~ y \in \mathcal Y.
\end{align}
Since $u$ is a weak efficient point  of $({\bf{\Phi}}_{2} \ominus_{gH}{\bf{\Phi}}_{1})$,
\begin{align}\label{04}
    & ({\bf{\Phi}}_{2} \ominus_{gH}{\bf{\Phi}}_{1})(u)\preceq({\bf{\Phi}}_{2} \ominus_{gH}{\bf{\Phi}}_{1})( y) ~\forall~ y \in \mathcal Y.
\end{align}
\begin{enumerate}[$\bullet$ \textbf{Case} 1.]
    \item If $w({\bf{\Phi}}_{1}(u)) \geq w({\bf{\Phi}}_{2}(u))$, then from the inequality (\ref{04}), for all $y \in \mathcal Y$, we have
    \begin{eqnarray}
     &&  [\overline {\phi}_{2}(u)-\overline {\phi}_{1}(u),\underline {\phi}_{2}(u) -\underline {\phi}_{1}(u)] \preceq[\overline {\phi}_{2}(y)-\overline{\phi}_{1}(y), \underline {\phi}_{2}(y)-\underline {\phi}_{1}(y)]\nonumber \\
     &\implies& \overline{\phi}_{1}(y)-\overline {\phi}_{1}(u) \le\overline{\phi}_{2}(y) -\overline{\phi}_{2}(u),~ \& ~\underline{\phi}_{1}(u)-\underline {\phi}_{1}(u) \le \underline {\phi}_{2}(y)-\underline {\phi}_{2}(u)\label{5_6}
     \end{eqnarray}
Now there arises two subcases.
\begin{enumerate}[$\bullet$ \textbf{Subcase} 1.]
\item If $\underline{\phi}_{1}(y)-\underline {\phi}_{1}(u) \le  \overline {\phi}_{1}(y)-\overline{\phi}_{1}(u) ,\\
\underline{\phi}_{1}(y)-\underline{\phi}_{1}(u)  \le \min \{ \underline{\phi}_{2}(y)-\underline{\phi}_{2}(u),\overline{\phi}_{2}(y) -\overline{\phi}_{2}(u)\}$ and  \\$\overline {\phi}_{1}(y)-\overline{\phi}_{1}(u) \le \max\{\underline{\phi}_{2}(y)-\underline {\phi}_{2}(u),\overline{\phi}_{2}(y) -\overline {\phi}_{2}(u)\}.$
Clearly we have
$[ \underline {\phi}_{1}(y)-\underline{\phi}_{1}(u) , \overline{\phi}_{1}(y)-\overline{\phi}_{1}(u)] \preceq ~[\min \{\underline {\phi}_{2}(y)-\underline{\phi}_{2}(u),\overline{\phi}_{2}(y) -\overline{\phi}_{2}(u)\},\\
~~~~~~~~~~~~~~~~~~~~~~~~~~~~~~~~~~~~~~~~~  \max \{\underline{\phi}_{2}(y)-\underline{\phi}_{2}(u),\overline{\phi}_{2}(y) -\overline {\phi}_{2}(u)\}].
 $
\item If $\overline{\phi}_{1}(y)-\overline {\phi}_{1}(u)  \le  \underline {\phi}_{1}(y)-\underline{\phi}_{1}(u) ,\\
\overline{\phi}_{1}(y)-\overline{\phi}_{1}(u)  \le \min \{ \underline{\phi}_{2}(y)-\underline{\phi}_{2}(u),\overline{\phi}_{2}(y) -\overline{\phi}_{2}(u)\}$ and
\\$\underline{\phi}_{1}(y)-\underline{\phi}_{1}(u) \le \max\{\underline{\phi}_{2}(y)-\underline{\phi}_{2}(u),\overline{\phi}_{2}(y) -\overline{\phi}_{2}(u)\}.$
Clearly we have
$ [ \overline {\phi}_{1}(y)-\overline {\phi}_{1}(u) , \underline {\phi}_{1}(y)-\underline{\phi}_{1}(u)]
 \preceq [\min \{\underline{\phi}_{2}(y)-\underline{\phi}_{2}(u),\overline {\phi}_{2}(y) -\overline {\phi}_{2}(u)\}, \\ ~~~~~~~~~~~~~~~~~~~~~~~~~~~~~~~~~~~~~~~~ \max \{\underline {\phi}_{2}(y)-\underline {\phi}_{2}(u),\overline {\phi}_{2}(y) -\overline {\phi}_{2}(u)\}]
.$
\end{enumerate}
  Combining $\textbf{Subcase}$ 1 and $\textbf{Subcase}$ 2, we have
\begin{align}\label{ghc}
{\bf{\Phi}}_{1}(y) \ominus_{gH}{\bf{\Phi}}_{1}(u) \preceq{\bf{\Phi}}_{2}(y) \ominus_{gH} {\bf{\Phi}}_{2}(u).
 \end{align}
     \item If $w({\bf{\Phi}}_{2}(u)) \geq w({\bf{\Phi}}_{1}(u))$, then from the inequality (\ref{04}), for all $y \in \mathcal Y$, we have
      \begin{eqnarray}
     &&  [\underline {\phi}_{2}(u)-\underline {\phi}_{1}(u),\overline {\phi}_{2}(u) -\overline {\phi}_{1}(u)]\preceq[\underline {\phi}_{2}(y)-\underline {\phi}_{1}(y), \overline {\phi}_{2}(y)-\overline {\phi}_{1}(y)]  \nonumber \\
    &\implies& \underline {\phi}_{1}(y)-\underline {\phi}_{1}(u) \le \underline {\phi}_{2}(y) -\underline {\phi}_{2}(u)~ \& ~\overline {\phi}_{1}(y)-\overline {\phi}_{1}(u) \le  \overline {\phi}_{2}(y)-\overline {\phi}_{2}(u). \label{3_2}
    \end{eqnarray}
By a similar manner as in $\textbf{Case}~1$, we have
  \begin{align*}
    {\bf{\Phi}}_{1}(y) \ominus_{gH}{\bf{\Phi}}_{1}(u)\preceq{\bf{\Phi}}_{2}(y) \ominus_{gH}{\bf{\Phi}}_{2}(u).
 \end{align*}
\end{enumerate}
Hence, in all cases, we have
 \begin{align}\label{bb}
   {\bf{\Phi}}_{1}(y) \ominus_{gH}{\bf{\Phi}}_{1}(u) \preceq{\bf{\Phi}}_{2}(y) \ominus_{gH}{\bf{\Phi}}_{2}(u).
 \end{align}
In view of (\ref{er}) and from (\ref{bb}), we get
\begin{align}
     \widehat{\textbf U^w}^{\top} \odot (y-u) \ominus_{gH} c \lVert y-u \rVert\preceq{\bf{\Phi}}_{2}(y) \ominus_{gH}{\bf{\Phi}}_{2}(u) ~\text{for all}~ y \in \mathcal Y, ~\text{by Lemma 2.3~(ii)~of ~\cite{anshika}}\nonumber.
    \end{align}
    which implies $(\widehat{\textbf U^w},c) \in \partial^{w}{\bf{\Phi}}_{2}(u)$. Hence, the result follows.
\end{proof}

\begin{note}
If we had taken an efficient solution of ${\bf{\Phi}}_{2} \ominus_{gH} {\bf{\Phi}}_{1}$ instead of a weak efficient solution, the additional condition
  $w({\bf{\Phi}}_{1}( y)) \geq w({\bf{\Phi}}_{2}( y))$ or $w({\bf {\Phi}}_{1}( y)) \leq w({\bf{\Phi}}_{2}( y))$ becomes essential for  Theorem \ref{plk} to hold. For instance, consider the IVFs ${\bf{\Phi}}_{1} : [-1,1] \rightarrow I(\mathbb{R})$ and ${\bf{\Phi}}_{2} : [-1,1] \rightarrow I(\mathbb{R})$ which are defined by
\[{\bf{\Phi}}_{1}(y)= \begin{cases}
 [y^3,y] & \text{ if } 0 \leq y \leq 1\\
 [4y,y] & \text{ if } -1\leq y < 0 \\
\end{cases}
~~\text{and}~~
{\bf{\Phi}}_{2}(y)=\begin{cases}
 [y^3,5y] & \text{ if } 0 \leq y \leq 1  \\
 [3y, 2y] & \text{ if } -1\leq y < 0,  \\
\end{cases}\] respectively. Now, according to Theorem \ref{plk},
\[({\bf{\Phi}}_{2} \ominus_{gH}{\bf{\Phi}}_{1})(y)= \begin{cases}
[0,4y] &  \text{ if } 0 \leq y \leq 1\\
[y, -y] &  \text{ if } -1 \leq y < 0\\
\end{cases}\] gets efficient solution at $0$ because $({\bf{\Phi}}_{2} \ominus_{gH}{\bf{\Phi}}_{1})(0) \preceq ({\bf{\Phi}}_{2} \ominus_{gH}{\bf{\Phi}}_{1})(y)$ for all $ y \in$ $[0,1]$ and $({\bf{\Phi}}_{2}\ominus_{gH}{\bf{\Phi}}_{1})(0)$ is not comparable with the values $({\bf{\Phi}}_{2} \ominus_{gH}{\bf{\Phi}}_{1})(y)$ for all $y \in [-1,0]$. It is not difficult to check that
\begin{align*}
 &\partial^{w}{\bf{\Phi}}_{1}(0)=\{({\textbf K^{w}_{1}}, c_{1}): [1,4] \preceq({\textbf K^{w}_{1}}\oplus c_{1} ), {\textbf K^{w}_{1}}\ominus_{gH} c_{1} \preceq [0,1]\}\\
\text{and}~~ &\partial^{w}{\bf {\Phi}}_{2}(0)=\{({\textbf K^{w}_{2}}, c_{2}): [2,3] \preceq {\textbf K^{w}_{2}}\oplus c_{2}, {\textbf K^{w}_{2}}\ominus_{gH} c_{2}  \preceq [0,5]\}.
\end{align*}
Here, we see that $\partial^{w}{\bf{\Phi}}_{1}(0)$ and $\partial^{w}{\bf{\Phi}}_{2}(0)$ are not comparable and at same time, we  notice that $w({\bf{\Phi}}_{2}(y)) \geq w({\bf{\Phi}}_{1}(y))$ on $[0,1]$ and $w({\bf{\Phi}}_{1}(y)) \geq w({\bf{\Phi}}_{2}(y))$ on $[-1,0]$.
\end{note}
%%%%%%%%%%%%%%%%%%%%%%%%%%%%%%%%%%%%%%%%%%%%%%%%%%%%%%%%%%%%%%%%%%%%%%%%%%%%%%%%%%%%%%%%%%%%%%%%%%
\begin{remark}
In Theorem \ref{plk}, the inclusion (\ref{dini2}) is necessary but not sufficient condition for weak efficient point of ${\bf {\Phi}}_{2}\ominus_{gH}{\bf{\Phi}}_{1}$. For instance, consider the IVFs ${\bf{\Phi}}_{1}: [-1,1] \rightarrow I(\mathbb{R})$ and~${\bf{\Phi}}_{2}: [-1,1] \rightarrow I(\mathbb{R})$ that are defined by  \[{\bf{\Phi}}_{1}(y)=\begin{cases}
[y^3, y]  & \text{if } 0 \leq y \leq 1 \\
[3y, 1.5 y]  & \text{if } -1 \leq y < 0 \\
\end{cases}
~~\text{and}~~
{\bf{\Phi}}_{2}(y)=\begin{cases}
 [y^3+y^2, 2y^2+y] & \text{if } 0 \leq y \leq 1\\
 [3y,2y] &\text{if } -1 \leq y < 0. \\
\end{cases} \]
We notice that $w({\bf{\Phi}}_{2}(y)) \geq w({\bf{\Phi}}_{1}(y))$ on $[0,1]$ and $w({\bf{\Phi}}_{2}(y)) \leq w(\boldsymbol {\Phi}_{1}(y))$ on $[-1,0]$. Note that
\begin{align*}
&\partial^{w}{\bf{\Phi}}_{1}(0)= \{({\textbf K^{w}_{1}}, c_{1}): [1.5,3] \preceq {\textbf K^{w}_{1}}\oplus c_{1}, {\textbf K^{w}_{1}}\ominus_{gH} c_{1}\preceq [0,1]\} \\
\text{ and } &\partial^{w}{\bf{\Phi}}_{2}(0) = \{({\textbf K^{w}_{2}},c_{2}): [2, 3] \preceq {\textbf K^{w}_{2}}\oplus c_{2}, {\textbf  K^{w}_{2}}\ominus_{gH} c_{2}\preceq [0,1] \}.
\end{align*}
Hence, $\partial^{w}{\bf{\Phi}}_{1}(0) \subset \partial^{w}{\bf{\Phi}}_{2}(0)$ but $0$ is not a weak efficient point of ${\bf{\Phi}}_{2}\ominus_{gH} {\bf{\Phi}}_{1}$ on $[-1,1]$.
\end{remark}

Next, we study a relation between augmented normal cone and $gH$-weak subdifferential. So, let us define the  augmented normal cone  to $\mathcal Y$  as below.
\begin{definition}(\emph{Augmented normal cone}). An \emph{ augmented normal cone} to $\mathcal Y$ at $u$ is
\begin{align*}
    \mathcal N^{c}_{\mathcal Y}(u)= \left\{ (\widehat{\textbf G}, c) \in I(\mathbb R)^n \times \mathbb R_{+}: \widehat{\textbf G}^\top \odot (y-u) \ominus_{gH} c\lVert y-u \rVert \preceq \textbf 0 ~\forall~y \in \mathcal Y\right\}.
\end{align*}
% Note that if there exists $\widehat{\textbf G} \in \mathcal Y$ such that $(\widehat{\textbf G} ,0)\in \mathcal N^{c}_{\mathcal Y}(u) $, then $\widehat{\textbf G}\in \mathcal N_{\mathcal Y}(u) $.
\end{definition}

\begin{theorem}\emph{(Optimality condition via augmented normal cone)}. \label{ybnm}
    An IVF ${\bf{\Psi}} : \mathcal{Y} \to I(\mathbb{R})$ attains weak efficient solution at $u$ if and only if $(\textbf 0,0) \in \partial^{w}{{\bf{\Psi}}}(u) \oplus \mathcal N^{c}_{\mathcal Y}(u)$, where  $(\textbf 0,0)$ denotes the zero of $I(\mathbb{R}) \times \mathbb{R}_{+}$.
\end{theorem}
%%%%%%%%%%%%%%%%%%%%%
\begin{proof}
Since  $u$ is a weak efficient point of ${\bf{\Psi}}$ on $\mathcal Y$,
\begin{align*}
   &  {\bf{\Psi}}(u)\preceq{\bf {\Psi}}(y) ~\forall ~  y \in \mathcal Y\nonumber\\
   \implies & \textbf 0\preceq{{\bf{\Psi}}}(y) \ominus_{gH}{\bf{\Psi}}( u)~\forall ~  y \in \mathcal Y ~~\text{by Lemma 2.1(ii)~ of~\cite{Ghosh2019derivative}}  \nonumber\\
    \implies &  (\textbf 0, 0) \in \partial^{w} {\bf{\Psi}}(u).
\end{align*}
Let $\delta_{\mathcal Y}: \mathcal Y \to I(\mathbb R)$ be an indicator function, defined by
$\delta_{\mathcal Y}(y)= \begin{cases}
\textbf 0, & \text{for} ~y \in \mathcal Y\\
\infty, & \text{for }~ y \notin \mathcal{Y}
\end{cases}$.
Since \[({\bf{\Psi}} \oplus \delta_{\mathcal Y})(y)=
\begin{cases}
{{\bf{\Psi}}}(y) & \text{ if }~ y \in \mathcal{Y} \\
 \infty &\text{ if }~ y \notin \mathcal{Y}, \\
\end{cases}\]
$(\textbf 0,0  )\in \partial^{w} {\bf{\Psi}}(u)= \partial^{w}({\bf {\Psi}} \oplus \delta_{\mathcal{Y}})(u)$.
It needs to show that $\partial^{w}({\bf{\Psi}} \oplus \delta_{\mathcal{Y}})(u) \subset \partial^{w} {\bf{\Psi}}(u) \oplus \mathcal{N}^{c}_{\mathcal Y}(u)$. To prove this,
let $\widehat{\textbf G^w} \in \partial^{w}({\bf{\Psi}}_{1}\oplus \delta_{\mathcal Y})(u) $.
Then,
\begin{align*}
    & \widehat{\textbf G^w}^{\top} \odot (y-u) \ominus_{gH} c\lVert y-u\rVert\preceq ({\bf{\Psi}}\oplus \delta_{\mathcal Y})(y) \ominus_{gH} ({\bf{\Psi}} \oplus \delta_{\mathcal Y})(u)  \nonumber\\
   \implies &\widehat{\textbf G^w}^{\top}\odot (y-u) \ominus_{gH} c \lVert y-u \rVert \preceq ({{\bf{\Psi}}}(y) \oplus \delta_{\mathcal Y}(y))\ominus_{gH} ({\bf{\Psi}}(u) \oplus \delta_{\mathcal Y}(u))   \nonumber\\
\implies &   \widehat{\textbf G^w}^{\top} \odot (y-u) \ominus_{gH} c \lVert y-u \rVert\preceq {{\bf{\Psi}}}(y) \ominus_{gH}  {\bf{\Psi}}(u) ,
\end{align*}
which implies  $\widehat{\textbf G^w} \in \partial^{w} {\bf{\Psi}}(u) \subset \partial^{w} {\bf{\Psi}}(u) \oplus \partial^{w} \delta_{\mathcal Y}(u)$, where \{($\textbf 0,0 )\} \subset \partial^{w}\delta_{\mathcal Y}(u)$.
Hence, $\widehat{\textbf G^w} \in \partial^{w}{\bf{\Psi}}(u) \oplus \partial^{w}\delta_{\mathcal Y}(u)=\partial^{w} {\bf{\Psi}}(u) \oplus \mathcal{N}^{c}_{\mathcal Y}(u)$.\\
To show the converse part, let $(\textbf 0, 0) \in \partial^{w}{\bf{\Psi}}(u) \oplus \mathcal{N}^{c}_{\mathcal Y}(u)=   \partial^{w} ({\bf{\Psi}}(u) \oplus  \delta_{\mathcal Y}(u))$.
Now, for any  $ y \in \mathcal Y$,\begin{align*}
    & \textbf 0 \odot (y-u) \ominus_{gH} 0 \lVert y-u \rVert\preceq ({\bf{\Psi}}( y) \oplus \delta_{\mathcal Y}(y)) \ominus_{gH} ({\bf{\Psi}}(u) \oplus  \delta_{\mathcal Y}(u))   \\
\text{or},~    & \textbf 0\preceq {{\bf{\Psi}}}(y) \ominus_{gH} {\bf{\Psi}}(u) \nonumber\\
\text{ or},~ & {\bf{\Psi}}(u) \preceq {{\bf{\Psi}}}(y) ~\text{by Lemma 2.1(ii)~of~\cite{Ghosh2019derivative}}.
\end{align*}
So, $u$ is a weak efficient solution of ${\bf{\Psi}}$.
\end{proof}

\section{Conclusion}
\label{section6}
In this paper, the concepts of $gH$-weak subdifferentials and $gH$-weak  subgradients (Definition \ref{etuhj}) for IVFs with illustrative examples have been provided. The $gH$-weak subdifferential set of an IVF has been found to be convex (Theorem \ref{opi}) and closed (Theorem \ref{opw}). We have further introduced a necessary and sufficient condition (Theorem \ref{hgfd}) for the set of $gH$-weak subgradients to be nonempty. We have derived the necessary optimality condition (Theorem \ref{dsa}) involving $gH$- Fr\'echet differential and $gH$-weak subdifferential for IVFs. We have derived a necessary optimality criterion for difference of two IVFs (Theorem \ref{tyr} and Theorem \ref{ybnm}). Towards the end of the paper, we have provided a necessary and sufficient condition for weak efficient solution in terms of two notions of augmented normal cone and $gH$-weak subdifferential.  \\

Continuing the present study,  in the forthcoming work we will attempt to solve  the following three problems.
\begin{enumerate}[$\bullet$]
\item Introducing a $gH$-weak subgradient algorithm which characterizes efficient  solutions for nonsmooth nonconvex interval optimization problems.

\item In future, we will take up the practical optimization problems to be solved by $gH$-weak subgradient algorithm.

\item Analogous to the notion of weak-stability for conventional optimization problems \cite{gasimov}, in future, one may attempt to extend the notion for the following IOP (P):
\begin{align*}
               \min &~~{{\bf{\Phi}}}(y)\\
                \text{subject to}&~~ g_{j}(y) \leq 0, ~j=1,2,\ldots,p\\
                &~~  y \in \mathcal{Y},
\end{align*}
where ${\bf{\Phi}} : \mathcal{Y} \rightarrow I(\mathbb{R})\cup \{-\infty,+\infty\}$ is an IVF and $g_{j}: \mathcal{Y} \rightarrow \mathbb{R}$ is a real-valued constraint,  $j=1,2,\ldots,p$, and the feasible set $C$ is
\[   C = \{y \in \mathbb{R}^n: y \in \mathcal Y,  ~g_{j}(y) \leq 0, j=1,2,\ldots, p  \}.
\]
To establish an interrelation between strong duality and weak-stability for (P), one may define the augmented Lagrange interval-valued function for (P) as follows. Let $J$ be an arbitrary index set, for which define
\begin{align*}
    &\mathbb R^{(J)}_{\lambda}:= \{ e \in \mathbb R^{(J)}: \lvert e_{j} \rvert \leq 1 , j \in J(\lambda)\}\\
~\text{and}~& \Lambda: = \left\{(\lambda, k) \in \mathbb R^{(J)}: \exists ~e
 \in  \mathbb R^{(J)}_{\lambda}, ke -\lambda \in \mathbb R^{(J)}_{+} \right\},
\end{align*}
where
\begin{align*}
 &   \mathbb R^{(J)}: = \{\lambda= (\lambda_{j})_{j \in J}: \lambda_{j}=0 ~\text{for all}~j \in J ~\text{but only finitely many}~\lambda_{j} \neq 0\},\\
 & J(\lambda): = \{j \in J : \lambda_{j}\neq 0 \}, ~\text{is a finite subset of}~ J~\\
 ~\text{and}~&~~ \mathbb R^{J}_{+}: =\{\lambda = (\lambda_{j})_{j \in J}\in \mathbb R^{(J)}: \lambda_{j} \ge 0, j \in J \}.
\end{align*}
For each $j \in J$, the augmented Lagrange interval-valued function for (P) can be  defined by
\begin{align*}
    \textbf L(y, \Lambda, k) = \textbf T(y) \ominus_{gH} \langle\lambda, (g_{j}(y))_{j} \rangle \oplus \boldsymbol{\beta} ((g_{j}(y)), \lambda,k ),
\end{align*}
where $\boldsymbol{\beta} (u, \lambda, k ) : \mathbb R^J \times \mathbb R^{(J)}\times \mathbb R_{+} \to \mathbb R$ is such that
\[\boldsymbol{\beta} (y, \lambda , k) = \begin{cases}
   \sup_{e \in \mathbb R^{(J)}_{\lambda}} \left\{\langle ke, u\rangle : ke -\lambda \in \mathbb R^{(J)}_{+}  \right\}~ & \text{if}~J(\lambda)\neq \emptyset,\\
    0 & \text{if}~J(\lambda)= \emptyset.
    \end{cases}
\]
 The dual of (P) can be found as
 \begin{align*}
     \max &~~~\inf~ \textbf L(x, \lambda, k)\\
     \text{subject to} &~~~ (\lambda, k) \in \Lambda.
\end{align*}
We will make an effort to reduce the duality gap by weak-stability property of the following perturbation function  $ {\bf{\Psi}}: \mathcal{Y} \times \mathbb{R}^n \to
 I(\mathbb{R}) \cup  \{+ \infty\}$ associated  to the IOP (P):
 \begin{align*}
 &  {\bf{\Psi}}(y, u)= \begin{cases}
                 {{\bf{\Phi}}}(y)~ & \text{if}~ y \in  \mathcal{Y} \subset \mathbb{R}^n~\text{and}~ ~g_{j}(y) \leq u_{j},~ \forall~ j =1,2,3, \ldots,p\\
                 + \infty~ & \text{otherwise},
               \end{cases}
\end{align*}
where $u =(u_1,u_2, \ldots, u_n)$ is called the perturbation vector. \\

\item One may also try to apply of the $gH$-weak subdifferential in the context of zero duality gap in IOPs and interval-valued differential equations. The  method for eliminating duality gap will be immediately applicable in the following areas:
\begin{itemize}
    \item two person zero-sum game \cite{Parpas}
    \item optimal solutions of control problems with first order differential equations   \cite{Picken}
    \item Hamilton-Jacobi field theory  \cite{Picken}
    \item difference of convex programming \cite{Gaop}.\\
\end{itemize}

\item The newly defined augmented normal cone and $gH$-weak subdifferential together lead to the thought of introducing supporting cones for set of intervals in the future. This new concept may be used later to describe conic gap, which may be a crucial property to  capturing the geometry of nonconvex set of intervals.
\end{enumerate}

\noindent

\appendix

\section{Proof of Lemma \ref{dr2}} \label{appendix_ind}
\begin{proof}
Let $\textbf{W}=[\underline w, \overline w],~ \textbf{Y}=[\underline y, \overline y]$ and $ ~\textbf{Z}=[\underline z, \overline z].$
 From the $gH$-difference, we have the following four possible cases:
 \begin{enumerate}[$\bullet$ \textbf{Case} 1.]
 \item Given that $ \epsilon \preceq (\textbf{W} \ominus_{gH} \textbf{Y}) \ominus_{gH} \textbf{Z}= [\underline w -\underline y-\underline z, \overline w-\overline y-\overline z]$ \label{case 11} .
 Since $ \underline w-\underline y  \geq \underline z + \epsilon $ and $\overline w -\overline y \geq \overline z + \epsilon$, we have $ \underline z + \epsilon \le \overline z + \epsilon \le \overline w-\overline y$. This implies $\underline z+ \epsilon \le \min \{\underline w-\underline y, \overline w-\overline y \}$. Also, $\overline z + \epsilon \le \overline w -\overline y \le \max \{ \underline w-\underline y, \overline w -\overline y\}.$ Clearly we have $[\underline z + \epsilon, \overline z + \epsilon] \preceq [\min\{ \underline w-\underline y, \overline w-\overline y  \}, \max \{\underline w-\underline y, \overline w-\overline y \} ]$ and hence $\textbf{Z}\oplus\epsilon\preceq\textbf{W} \ominus_{gH} \textbf{Y} $.
 \item $(\textbf{W} \ominus_{gH} \textbf{Y}) \ominus_{gH} \textbf{Z}=[\overline w-\overline y-\overline z, \underline w-\underline y -\underline z]$. Thus, the proof is straightforward and identical to \textbf{Case} \ref{case 11}.
 \item \label{cvb}$(\textbf{W} \ominus_{gH} \textbf{Y}) \ominus_{gH} \textbf{Z}=[\overline w-\overline y-\underline z, \underline w-\underline y-\overline z]$. Since $\overline w -\overline y \geq \underline z + \epsilon,  \underline w -\underline y \geq \overline z + \epsilon$, we have  $ \underline z + \epsilon \le \overline z + \epsilon \le \underline w-\underline y$. This implies $\underline z+ \epsilon \le \min \{\underline w-\underline y, \overline w-\overline y \}$. Also, $\overline z + \epsilon \le \underline w -\underline y \le \max \{ \underline w-\underline y, \overline w -\overline y\}.$
 Clearly we have $   [\underline z + \epsilon, \overline z +\epsilon]   \preceq [\min\{ \underline w-\underline y, \overline w-\overline y  \}, \max \{\underline w-\underline y, \overline w-\overline y \} ] $ and hence $\textbf{Z}\oplus  \epsilon \preceq\textbf{W} \ominus_{gH} \textbf{Y} $.
 \item$(\textbf{W} \ominus_{gH} \textbf{Y}) \ominus_{gH} \textbf{Z}=[\underline w-\underline y-\overline z, \overline w-\overline y-\underline z]$. Thus, the proof is identical to \textbf{Case} \ref{cvb}.

 \end{enumerate}
\end{proof}	

\section{Proof of Lemma \ref{gh2}} \label{appendix_inf}
\begin{proof}
Let $\textbf X=[\underline x, \overline x], \textbf Y= [\underline y, \overline y]$, $ \textbf Z= [\underline z, \overline z]$~
and ~$ \textbf W=[\underline w, \overline w]$. Then,
\begin{align}
&(\textbf X \oplus \textbf Y) \ominus_{gH} (\textbf Z \oplus \textbf W) \nonumber\\
= ~& [\min \{\underline x+\underline y-\underline z-\underline w, \overline x+\overline y-\overline z-\overline w\}, \max \{\underline x+\underline y-\underline z-\underline w, \overline x+\overline y-\overline z-\overline w\}] \nonumber\\
=~& [\min\{\underline x - \underline z + \underline y - \underline w, \overline x- \overline z + \overline y - \overline w\}, \max\{  \underline x - \underline z + \underline y - \underline w, \overline x- \overline z + \overline y - \overline w    \}].\label{tuncjvhj_1}
\end{align}
We have
\begin{align}
  & \min \{ \underline x - \underline z + \underline y - \underline w , \overline x - \overline z + \overline y - \overline w \} \ge \min \{ \underline x - \underline z, \overline x - \overline z\} + \min \{ \underline y - \underline w, \overline y - \overline w\} \label{ttfftu_1} \\
 ~\text{and}~ & \max \{ \underline x - \underline z + \underline y - \underline w , \overline x - \overline z + \overline y - \overline w \} \le \max \{ \underline x - \underline z, \overline x - \overline z \} + \max \{\underline y - \underline w, \overline y - \overline w \} \label{ttfftu_2}.
\end{align}
 By  (\ref{ttfftu_1}) and (\ref{ttfftu_2}),  from (\ref{tuncjvhj_1}), we write
 \begin{align*}
& (\textbf X \oplus \textbf Y) \ominus_{gH} (\textbf Z \oplus \textbf W) \nonumber\\
 = ~& [ \min\{\underline x - \underline z + \underline y - \underline w, \overline x- \overline z + \overline y - \overline w\}, \max\{  \underline x - \underline z + \underline y - \underline w, \overline x- \overline z + \overline y - \overline w    \}] \\
 \subseteq ~ & [\min \{\underline x- \underline z, \overline x -\overline z \}+ \min  \{ \underline y - \underline w, \overline y - \overline w\}, \max \{\underline x- \underline z, \overline x -\overline z \} + \max \{\underline y - \underline w, \overline y - \overline w \}]   \\
 = ~& [\min \{\underline x- \underline z, \overline x -\overline z \},\max \{\underline x- \underline z, \overline x -\overline z \} ]+ [\min  \{ \underline y - \underline w, \overline y - \overline w\},  \max \{\underline y - \underline w, \overline y - \overline w \} ]\\
 = ~& (\textbf X \ominus_{gH} \textbf Z) \oplus (\textbf Y \ominus_{gH} \textbf W).
 \end{align*}
\end{proof}

\section{Proof of Lemma \ref{ldr1}} \label{appendix_inq}
\begin{proof}
Let $\textbf{W}=[\underline w, \overline w],  \textbf{Y}= [\underline y, \overline y]$ and $ \textbf{Z}= [\underline z, \overline z]$. Then, $-1 \odot \textbf{W}= [-\overline w, -\underline w], -1 \odot \textbf{Y}=[-\overline y, -\underline y], -1 \odot \textbf{Z}=[-\overline z, -\underline z]$. \\
From Definition of $gH$-difference of two intervals, we have\\
either $-1 \odot \textbf{W} \ominus_{gH} -1 \odot \textbf{Y}= [\overline y -\overline w, \underline y -\underline w]$ or $-1 \odot \textbf{W} \ominus_{gH} -1 \odot \textbf{Y}=[\underline y -\underline w, \overline y- \overline w]$.\\
Then, one of the following holds true :
\begin{enumerate}[(a)]
 \item $((-1 \odot \textbf{W})\ominus_{gH} (-1 \odot\textbf Y) )\ominus_{gH} (-1 \odot \textbf{Z})=[\overline y-\overline w+\overline z, \underline y- \underline w+\underline z]$
    \item $((-1 \odot\textbf{W})\ominus_{gH} (-1\odot \textbf Y) )\ominus_{gH} (-1 \odot \textbf{Z})=[\underline y-\underline w+\underline z, \overline y-\overline w+\overline z]$
    \item $((-1 \odot \textbf{W})\ominus_{gH} (-1 \odot \textbf Y)) \ominus_{gH} (-1 \odot \textbf{Z})=[\underline y-\underline w+\overline z,\overline y-\overline w+\underline z]$
    \item $((-1 \odot \textbf{W})\ominus_{gH} (-1 \odot\textbf Y))\ominus_{gH} (-1 \odot \textbf{Z})=[\overline y-\overline w+\underline z, \underline y-\underline w+\overline z]$.
\end{enumerate}
From this, we have
\begin{enumerate}[(a)]
 \item $ \textbf 0 \ominus_{gH}\{((-1 \odot \textbf{W})\ominus_{gH} (-1 \odot\textbf Y) )\ominus_{gH} (-1 \odot \textbf{Z})\}=[\underline w-\underline y-\underline z, \overline w-\overline y-\overline z]$
    \item $ \textbf 0 \ominus_{gH}\{((-1 \odot\textbf{W})\ominus_{gH} (-1\odot \textbf Y) )\ominus_{gH} (-1 \odot \textbf{Z})\}=[\overline w-\overline y-\overline z,\underline w-\underline y-\underline z]$
    \item $ \textbf 0\ominus_{gH}\{((-1 \odot \textbf{W})\ominus_{gH} (-1 \odot \textbf Y) )\ominus_{gH} (-1 \odot \textbf{Z})\}=[\overline w-\overline y -\underline z,\underline w-\underline y-\overline z]$
    \item $ \textbf 0 \ominus_{gH}\{((-1 \odot \textbf{W})\ominus_{gH} (-1 \odot\textbf Y) )\ominus_{gH} (-1 \odot \textbf{Z})\}=[\underline w-\underline y-\overline z,\overline w-\overline y-\underline z]$.
\end{enumerate}
On the other hand, \begin{enumerate}[(a)]
 \item $(\textbf{W}\ominus_{gH} \textbf{Y}) \ominus_{gH}  \textbf{Z}=[\underline w-\underline y-\underline z, \overline w-\overline y-\overline z]$
    \item $(\textbf{W}\ominus_{gH} \textbf{Y}) \ominus_{gH}  \textbf{Z}=[\overline w-\overline y-\overline z,\underline w-\underline y-\underline z]$
    \item $(\textbf{W}\ominus_{gH} \textbf{Y}) \ominus_{gH}  \textbf{Z}=[\overline w-\overline y -\underline z,\underline w-\underline y-\overline z]$
    \item $(\textbf{W}\ominus_{gH} \textbf{Y}) \ominus_{gH}  \textbf{Z}=[\underline w-\underline y-\overline z,\overline w-\overline y-\underline z].$
\end{enumerate}
%So, \begin{enumerate}[(a)]
%\item $\ominus_{gH}(\textbf{X}\ominus_{gH} \textbf{Y} \ominus_{gH} (- %\textbf{Z}))=[\overline y-\overline x-\underline z, \underline y-\underline x %-\overline z]$
% \item$\ominus_{gH}(\textbf{X}\ominus_{gH} \textbf{Y} \ominus_{gH} (- %\textbf{Z}))=[\underline y-\underline x -\overline z, \overline y-\overline %x-\underline z]$
%\item $\ominus_{gH}(\textbf{X}\ominus_{gH} \textbf{Y} \ominus_{gH} (- %\textbf{Z}))=[\underline y-\underline x -\underline z, \overline y-\overline %x-\overline z]$
%\item $\ominus_{gH}(\textbf{X}\ominus_{gH} \textbf{Y} \ominus_{gH} (- %\textbf{Z}))= [\overline y-\overline x-\overline z, \underline y-\underline %x-\underline z]$.
%\end{enumerate}
Hence, the desired result follows.
\end{proof}
\section{Proof of Lemma \ref{yure}} \label{appendix_ins}
\begin{proof}
Let $\textbf{X}=[\underline x, \overline x], \textbf{Y} =[\underline y, \overline y]$ and $ \textbf{Z}=[\underline z, \overline z]$.
\begin{enumerate}[(i)]
\item Let us consider the following four representations:
\begin{enumerate}[(a)]
    \item $(\textbf{X} \ominus_{gH}\textbf{Y} )\ominus_{gH} \textbf{Z}= [\underline x-\underline y-\underline z, \overline x-\overline y-\overline z]$,
    \item $(\textbf{X} \ominus_{gH}\textbf{Y} )\ominus_{gH} \textbf{Z}=[\underline x-\underline y-\overline z, \overline x-\overline y-\underline z]$,
    \item $(\textbf{X} \ominus_{gH}\textbf{Y} )\ominus_{gH} \textbf{Z}=[\overline x-\overline y-\underline z, \underline x-\underline y-\overline z]$,
    \item $(\textbf{X} \ominus_{gH}\textbf{Y} )\ominus_{gH} \textbf{Z}=[\overline x-\overline y-\overline z, \underline x-\underline y-\underline z]$.
\end{enumerate}
\begin{enumerate}[$\bullet$ \textbf{Case} 1.]
\item  Given that $\textbf 0 \preceq \textbf{X} \ominus_{gH} \textbf{Y} $.  %\label{pp}
Then
we have
\begin{align}
& 0 \le \underline x-\underline y ~ \text{and} ~0 \le \overline x-  \overline y \nonumber \\
\implies & 0-\underline z  \leq \underline x -\underline y-  \underline z ~\text{and}~ 0-\overline z \leq \overline x -\overline y- \overline z \nonumber \\
\implies & [0-\underline z,0-\overline z] \preceq [\underline x-\underline y-\underline z, \overline x-\overline y-\overline z] \label{gyuufxfxg}.
\end{align}
So, from (\ref{gyuufxfxg}), we have $ \textbf 0 \ominus_{gH} \textbf{Z} \preceq (\textbf{X}\ominus_{gH}\textbf{Y}) \ominus_{gH} \textbf{Z}$  .
\item Similarly, we will arrive at this conclusion %\label{gyuufxfxg}.
  %  \begin{align}\label{gyuufxfxg}
  %      [0-\overline z, 0-\underline z] \preceq [\underline x-\underline y-\overline z, %\overline x-\overline y-\underline z]
   %     \end{align}
 (\ref{gyuufxfxg}). So, from  (\ref{gyuufxfxg}), we have $\textbf 0 \ominus_{gH} \textbf{Z} \preceq (\textbf{X}\ominus_{gH}\textbf{Y}) \ominus_{gH} \textbf{Z}$. \label{zsdf}
  \item  This case can be proved  by using same  steps as \textbf{Case 1}.
  %\[[0 -\underline z, 0-\overline z] \preceq [\overline x-\overline y-\underline z, \underline x-\underline y-\overline z]\]
    \item  This case can be proved  by using same  steps as \textbf{Case 2}.
    % \[[0- \overline z, 0-\underline z] \preceq [\overline x-\overline y-\overline z, \underline x-\underline y-\underline z] \] that helps us to acquire the required inequality.
\end{enumerate}
%Further, we take norm on both sides of inequality and ultimately, we will complete our task.
\item Let $ \textbf{W}=[\underline w, \overline w]$. By the definition of $gH$-difference, there may be the following four cases.
\begin{enumerate}[(a)]
    \item $(\textbf{X} \ominus_{gH}\textbf{Y} )\ominus_{gH} \textbf{W}= [\underline x-\underline y-\underline w, \overline x-\overline y-\overline w]$
    \item$(\textbf{X} \ominus_{gH}\textbf{Y} )\ominus_{gH} \textbf{W}=[\underline x-\underline y-\overline w, \overline x-\overline y-\underline w]$
    \item $(\textbf{X} \ominus_{gH}\textbf{Y} )\ominus_{gH} \textbf{W}=[\overline x-\overline y-\underline w, \underline x-\underline y-\overline w]$
    \item $(\textbf{X} \ominus_{gH}\textbf{Y} )\ominus_{gH} \textbf{W}=[\overline x-\overline y-\overline w, \underline x-\underline y-\underline w]$.
\end{enumerate}
 The following two cases are needed to consider for the representation of these above four cases.
\begin{enumerate}[$\bullet$ \textbf{Case} 1.]
\item Since $\textbf{Z} \preceq \textbf{X} \ominus_{gH} \textbf{Y}$, we have
\begin{align}
    & \underline{z} \leq \underline x -\underline y ~\text{and}~ \overline{z} \leq \overline{x}-\overline{y} \nonumber\\
~\implies~    & \underline z -\underline w \leq  \underline x -\underline y-\underline w ~\text{and}~ \overline{z}-\overline{w} \leq \overline{x}-\overline{y}-\overline{w}
 \nonumber\\
\implies~& \text{either}~  [\underline z -\underline w,\overline{z}-\overline{w}] \preceq [\underline x -\underline y-\underline w,\overline{x}-\overline{y}-\overline{w}]\label{qwd}\\ ~\text{or}~& \label{dfg} [\overline{z}-\overline{w},\underline z -\underline w] \preceq [\overline{x}-\overline{y}-\overline{w},\underline z -\underline y-\underline w]
\end{align}
From (\ref{qwd}) and (\ref{dfg}), we have $\textbf{Z} \ominus_{gH} \textbf{W} \preceq (\textbf{X}\ominus_{gH} \textbf{Y})\ominus_{gH} \textbf{W}$.
\item Similarly, at the last step, we  have
\begin{align} \label{ghf}
\text{either}~  [\underline z -\underline w,\overline{z}-\overline{w}] \preceq [\overline x -\overline y-\underline w,\underline{x}-\underline{y}-\overline{w}]\\ ~\text{or}~ \label{sad} [\overline{z}-\overline{w},\underline z -\underline w] \preceq [\underline{x}-\underline{y}-\overline{w},\overline x -\overline y-\underline w]
\end{align}
From (\ref{ghf}) and (\ref{sad}), we have $\textbf{Z} \ominus_{gH} \textbf{W} \preceq (\textbf{X}\ominus_{gH} \textbf{Y})\ominus_{gH} \textbf{W}$.
\end{enumerate}
\item Given that $\textbf X \ominus_{gH} \textbf Y \preceq [L,L]$.
From the formula of $gH$-difference of intervals,~
\begin{align*}
 &\underline x- \underline y \leq L ~~\text{and}~~ \overline x -\overline y \leq L
  \nonumber \\
  \implies &  -L  \leq \underline y-\underline x, -L \leq \overline y -\overline x \\
 \implies &~ ~\text{either} ~~[-L,-L] \preceq [\underline y-\underline x, \overline y-\overline x ]  \\
  &\text{or}~ [-L,-L] \preceq [\overline y-\overline x, \underline y-\underline x].
  %\label{fty}
  \end{align*}
  Hence, $[-L,-L] \preceq \textbf Y \ominus_{gH} \textbf X$.
\item Given that $[-\gamma,-\gamma] \preceq \textbf X \ominus_{gH} \textbf Y$. From the formula of $gH$-difference of intervals,
\begin{align*}
   &-\gamma \leq \underline x- \underline y ~~\text{and}~~ -\gamma \leq \overline x-\overline y \nonumber\\
   \implies&\underline y-\gamma \leq  \underline x ~\text{and}~ \overline y-\gamma \leq \overline x \\
  \implies & [\underline y-\gamma, \overline y-\gamma] \preceq [\underline x, \overline x]. \end{align*}
 Hence,    $ \textbf Y \ominus_{gH} [\gamma,\gamma] \preceq \textbf X .$
\item Given that
 $\textbf Z \preceq \textbf X \oplus \textbf Y$. Then,
 $$ \begin{aligned}
 & [\underline z, \overline z]\preceq[\underline x, \overline x] \oplus [\underline y,\overline y] \\
 \implies & \underline z \leq \underline x +\underline y, \overline z \leq \overline x +\overline y\\
 \implies& \underline z-\underline y \leq \underline x, \overline z-\overline y \leq \overline x\\
 \implies & [\underline z-\underline y,\overline z-\overline y ] \preceq [\underline x, \overline x].\end{aligned}$$
 Hence, $\textbf Z \ominus_{gH}\textbf{Y} \preceq \textbf X.$
\end{enumerate}
\end{proof}

\section{Proof of Lemma \ref{yuv}} \label{appendix_iou}
\begin{proof}
Let $y^{\top} \odot \widehat{\textbf C}=\textbf D$ and $\textbf D =[\underline d, \overline d]$. Note that \begin{align}\label{oiu}
    \lVert \textbf D \rVert_{I(\mathbb{R})}= \max\{\lvert \underline d\rvert ,\lvert \overline d \rvert \}.
\end{align}
On the other hand,
\begin{align}
\lVert \textbf D \rVert_{I(\mathbb{R})} =~& \lVert y_{1} \odot \textbf C_{1} \oplus y_{2} \odot \textbf C_{2} \oplus \cdots \oplus y_{n} \odot \textbf C_{n}\rVert_{I(\mathbb{R})}\nonumber\\
 \leq ~ & \lVert  y_{1} \odot \textbf C_{1} \rVert_{I(\mathbb{R})} + \lVert  y_{2} \odot \textbf C_{2}\rVert_{I(\mathbb{R})}+ \cdots+y_{n} \odot \textbf C_{n} \rVert_{I(\mathbb{R})} \nonumber\\
 = ~& \lvert y_{1}\rvert \lVert \textbf C_{1} \rVert _{I(\mathbb{R})} \oplus \lvert y_{2}\rvert \lVert \textbf C_{2} \rVert _{I(\mathbb{R})}+ \cdots
 +\lvert y_{n}\rvert \lVert \textbf C_{n} \rVert _{I(\mathbb{R})}\nonumber\\
 \leq ~ & \lVert y \rVert\sum_{i=1}^{n} \lVert \textbf C_{i} \rVert_{I(\mathbb{R})} \nonumber\\
 \label{fgh}=~& \lVert y\rVert \lVert \widehat{\textbf C} \rVert_{I(\mathbb{R})^n}.
\end{align}
Then, taking into account (\ref{oiu}) and (\ref{fgh}), we obtain
 \begin{eqnarray*}
     && \lvert \underline d \rvert \leq  \lVert y \rVert  \lvert \widehat{\textbf C}\rVert_{I(\mathbb{R})^n} ~\text{and}~ \lvert \overline d \rvert \leq \lVert y \rVert  \lVert \widehat{\textbf C}\rVert_{I(\mathbb{R})^n}\\
   &\implies & - \lVert y \rVert \lVert \widehat{\textbf C}\rVert_{I(\mathbb{R})^n} \leq \underline d ~\text{and}~ - \lVert y \rVert \lVert \widehat{\textbf C}\rVert_{I(\mathbb{R})^n} \leq \overline d \\
   &\implies & -\lVert y \rVert \lVert \widehat{\textbf C}\rVert_{I(\mathbb{R})^n} \leq \lvert \underline d \rvert ~\text{and}~ - \lVert y \rVert \lVert \widehat{\textbf C}\rVert_{I(\mathbb{R})^n} \leq \lvert \overline d \rvert \\
  &\implies&  - \lVert y \rVert \lVert \widehat{\textbf C}\rVert_{I(\mathbb{R})^n} \leq \max \{\lvert \underline d \rvert , \lvert \overline d \rvert \} \\
    &\implies&  - \lVert y \rVert \lVert \widehat{\textbf C}\rVert_{I(\mathbb{R})^n} \leq\lVert \textbf D \rVert_{I(\mathbb{R})}
  % &\implies&  - \lVert y \rVert \lVert \widehat{\textbf C}\rVert_{I(\mathbb{R})^n} \leq \lVert y^{T} \odot \widehat{\textbf C} \rVert_{I(\mathbb{R})}.
 \end{eqnarray*}
 Thus, we arrived at the desired result.
\end{proof}

\noindent

%%%%%%%%%%%%%%%%%%%%%%%%%%%%%%%%%%%%%%%%%%%%%%%%%%%%%%%%%%%%%%%%%
\section*{Funding}
D. Ghosh: MATRICS from SERB, India, with file number  MTR/2021/000696. \\
S. Ghosh acknowledges a research fellowship from University Grant Commission, India with file number  16-9(June2019)/2019(NET/CSIR).

\section*{Conflict of interest}
The authors declare that they have no known competing financial interests or personal relationships that could
have appeared to influence the work reported in this paper.

\section*{Availability of data and materials}
Not applicable.

\section*{Code availability}

Not applicable.

\section*{Authors' contributions}

All authors contributed to the study conception and analysis. All authors read and approved the
final manuscript.

% \section*{Acknowledgement}
% The first author acknowledges a research fellowship from University Grants Commission (Government of India). D. Ghosh puts thanks to the research grant MATRICS (MAT/ 2021/ 000696) by SERB, India.

%%%%%%%%%%%%%%%%%%%%%%%%%%%%%%%%%%%%%%%%%%%%%%%%%%%%%%%%%%

\end{document}